\documentclass[a4paper,11pt]{amsart}
\title{Stable pair invariants on Calabi-Yau 
threefolds 
containing $\mathbb{P}^2$}
\date{}
\author{Yukinobu Toda}

\usepackage{amscd}
\usepackage{amsmath}
\usepackage{amssymb}
\usepackage{amsthm}
\usepackage{float}
\usepackage[dvips]{graphicx}

\usepackage[all,ps,dvips]{xy}

\usepackage{array}
\usepackage{amscd}
\usepackage[all]{xy}

\DeclareFontFamily{U}{rsfs}{%
\skewchar\font127}
\DeclareFontShape{U}{rsfs}{m}{n}{%
<-6>rsfs5<6-8.5>rsfs7<8.5->rsfs10}{}
\DeclareSymbolFont{rsfs}{U}{rsfs}{m}{n}
\DeclareSymbolFontAlphabet
{\mathrsfs}{rsfs}
\DeclareRobustCommand*\rsfs{%
\@fontswitch\relax\mathrsfs}

\theoremstyle{plain}
\newtheorem{thm}{Theorem}[section]
\newtheorem{prop}[thm]{Proposition}
\newtheorem{lem}[thm]{Lemma}

\newtheorem{defi}[thm]{Definition}
\newtheorem{rmk}[thm]{Remark}
\newtheorem{cor}[thm]{Corollary}

\newtheorem{prop-defi}[thm]{Proposition-Definition}
\newtheorem{thm-defi}[thm]{Theorem-Definition}
\newtheorem{lem-defi}[thm]{Lemma-Definition}

\newtheorem{question}[thm]{Question}

\newtheorem{conj}[thm]{Conjecture}
\newtheorem{exam}[thm]{Example}

\newdimen\argwidth
\def\db[#1\db]{
 \setbox0=\hbox{$#1$}\argwidth=\wd0
 \setbox0=\hbox{$\left[\box0\right]$}
  \advance\argwidth by -\wd0
 \left[\kern.3\argwidth\box0 \kern.3\argwidth\right]}

\newcommand{\aA}{\mathcal{A}}
\newcommand{\bB}{\mathcal{B}}
\newcommand{\cC}{\mathcal{C}}
\newcommand{\dD}{\mathcal{D}}
\newcommand{\eE}{\mathcal{E}}
\newcommand{\fF}{\mathcal{F}}

\newcommand{\hH}{\mathcal{H}}

\newcommand{\lL}{\mathcal{L}}
\newcommand{\mM}{\mathcal{M}}

\newcommand{\oO}{\mathcal{O}}

\newcommand{\tT}{\mathcal{T}}
\newcommand{\uU}{\mathcal{U}}

\newcommand{\xX}{\mathcal{X}}
\newcommand{\yY}{\mathcal{Y}}

\newcommand{\Supp}{\mathop{\rm Supp}\nolimits}
\newcommand{\Hom}{\mathop{\rm Hom}\nolimits}

\newcommand{\dR}{\mathbf{R}}
\newcommand{\dL}{\mathbf{L}}
\newcommand{\NS}{\mathop{\rm NS}\nolimits}

\newcommand{\Pic}{\mathop{\rm Pic}\nolimits}

\newcommand{\ch}{\mathop{\rm ch}\nolimits}

\newcommand{\Ext}{\mathop{\rm Ext}\nolimits}
\newcommand{\Spec}{\mathop{\rm Spec}\nolimits}
\newcommand{\rank}{\mathop{\rm rank}\nolimits}
\newcommand{\Coh}{\mathop{\rm Coh}\nolimits}

\newcommand{\cneq}{\mathrel{\raise.095ex\hbox{:}\mkern-4.2mu=}}
\newcommand{\eqcn}{\mathrel{=\mkern-4.5mu\raise.095ex\hbox{:}}}

\newcommand{\Cok}{\mathop{\rm Cok}\nolimits}

\newcommand{\Aut}{\mathop{\rm Aut}\nolimits}

\newcommand{\Stab}{\mathop{\rm Stab}\nolimits}

\newcommand{\DT}{\mathop{\rm DT}\nolimits}
\newcommand{\PT}{\mathop{\rm PT}\nolimits}

\newcommand{\modu}{\mathop{\rm mod}\nolimits}
\newcommand{\End}{\mathop{\rm End}\nolimits}

\newcommand{\Imm}{\mathop{\rm Im}\nolimits}

\newcommand{\GL}{\mathop{\rm GL}\nolimits}

\newcommand{\cl}{\mathop{\rm cl}\nolimits}

\begin{document}
\maketitle

\begin{abstract}
We relate Pandharipande-Thomas 
stable pair invariants 
on Calabi-Yau 3-folds containing the projective plane 
with 
those on the derived equivalent orbifolds
via wall-crossing method. 
The difference is described by 
generalized Donaldson-Thomas invariants counting semistable sheaves
on the local projective plane, 
whose generating series form
theta type series for indefinite lattices. 
Our result also derives 
 non-trivial constraints
among stable pair invariants 
on such Calabi-Yau 3-folds
caused by Seidel-Thomas twist. 
\end{abstract}

\section{Introduction}
\subsection{Motivation}
It is an important subject to 
count algebraic curves on Calabi-Yau 3-folds, 
or more generally on CY3
orbifolds\footnote{In this paper, an orbifold 
means a smooth Deligne-Mumford stack.}, 
in connection with 
string theory. 
So far at least three curve counting theories have 
been proposed
and studied: Gromov-Witten (GW) theory~\cite{BGW}, 
Donaldson-Thomas (DT) theory~\cite{Thom}, \cite{MNOP} 
and Pandharipande-Thomas (PT) stable pair theory~\cite{PT}. 
It was conjectured, and proved in many cases, that these
theories are equivalent: 
the equivalence of DT and PT theories 
was proved in~\cite{BrH}, \cite{Tcurve1}, \cite{StTh}
using Hall algebras,
and the equivalence of GW and PT theories 
was proved by Pandharipande-Pixton~\cite{PP}
for many Calabi-Yau 3-folds including quintic 3-folds
using degenerations and torus 
localizations. 

On the other hand, the derived category of 
coherent sheaves $D^b \Coh(X)$
on a Calabi-Yau 3-fold $X$ 
is also an important mathematical subject, 
due to its role in Kontsevich's
Homological mirror symmetry conjecture~\cite{Kon}. 
It
was suggested by Pandharipande-Thomas~\cite{PT}
that the derived category also plays a crucial role in 
curve counting, 
as their stable pair invariants count
two term complexes
\begin{align*}
(\oO_X \stackrel{s}{\to} F) \in D^b \Coh(X) 
\end{align*}
where $F$ is a pure one 
dimensional sheaf and $s$ is surjective in dimension one. 
In this paper, we 
concern how symmetries in 
the derived categories 
affect stable pair invariants. 
More precisely, we 
are interested in the 
following questions:
\begin{question}\label{quest}
\begin{enumerate}

\item 
How stable pair 
invariants on two
Calabi-Yau 3-folds or orbifolds
are related, if they have
equivalent derived categories ?
\item
How stable pair 
invariants on a Calabi-Yau 3-fold
are constrained, due to the presence of 
non-trivial autoequivalences of the derived category ?
\end{enumerate}
\end{question}
The purpose of this paper 
is to study 
Question~\ref{quest}
for stable pair invariants 
on Calabi-Yau 3-folds $X$ which 
contain $\mathbb{P}^2$, and their derived 
equivalent CY3 orbifolds $\yY$. 
Our results include
new kinds of
progress on Question~\ref{quest}: 
(i) 
relation of stable pair invariants 
on $X$ and $\yY$, where 
$\yY$ does \textit{not} satisfy the 
Hard-Lefschetz (HL) condition\footnote{
The HL condition on a CY3 orbifold
is equivalent to that 
the crepant resolution of its coarse 
moduli space
has at most one dimensional fibers.}
(ii) constraints of stable pair 
invariants on $X$ caused by Seidel-Thomas twist~\cite{ST}. 
The relation of our work 
with the existing works 
will be discussed in 
Subsection~\ref{subsec:related}. 

\subsection{Main result}
Let $X$ be a smooth projective 
Calabi-Yau 3-fold which 
contains a divisor 
\begin{align*}
\mathbb{P}^2 \cong D \subset X.
\end{align*}
We have two phenomena 
related to (i) and (ii) in 
Question~\ref{quest}: 
\begin{enumerate}
\item
The divisor $D$ is contracted by a birational morphism 
$f \colon X\to Y$
to an orbifold singularity
with type 
$\frac{1}{3}(1, 1, 1)$. 
The associated smooth Deligne-Mumford
stack $\yY \to Y$ is 
derived equivalent to $X$
\begin{align}\notag
\Phi \colon
D^b \Coh(\yY) \stackrel{\sim}{\to} D^b \Coh(X). 
\end{align}
\item
The object $\oO_{D}$ is a spherical object
in $D^b \Coh(X)$, 
and we have the associated 
autoequivalence called 
\textit{Seidel-Thomas twist}~\cite{ST}
\begin{align}\notag
\mathrm{ST}_{\oO_{D}} \colon D^b \Coh(X)
\stackrel{\sim}{\to} D^b \Coh(X). 
\end{align}
\end{enumerate}
Contrary to the 3-fold
flop case as in~\cite{Tcurve2}, \cite{Cala}, 
curves on $\yY$, $X$ may be 
transformed to objects with 
two dimensional supports under
the equivalence $\Phi$, $\mathrm{ST}_{\oO_D}$, 
respectively. 
In order to deal with this issue, we 
also involve generalized DT invariants~\cite{JS}, \cite{K-S}
\begin{align}\label{intro:rcm}
\DT(r, c, m) \in \mathbb{Q}
\end{align}
 on 
the non-compact Calabi-Yau 3-fold
$\pi \colon \omega_{\mathbb{P}^2} \to \mathbb{P}^2$. 
The invariant (\ref{intro:rcm})
counts semistable sheaves $E$ on $\omega_{\mathbb{P}^2}$
satisfying
\begin{align*}
\rank(\pi_{\ast}E)=r, \ c_1(\pi_{\ast}E)=c, \ 
\ch_2(\pi_{\ast}E)=m. 
\end{align*}
The following is a rough statement of our main result: 
\begin{thm}\emph{(Theorem~\ref{thm:main}, Theorem~\ref{thm:const})}
\label{thm:intro}
Assuming Conjecture~\ref{conj:critical}
below, we have the following: 
\begin{enumerate}
\item
The stable pair invariants on $\yY$ are
described as explicit polynomials of 
stable pair invariants on $X$ and generalized DT invariants (\ref{intro:rcm})
 on 
$\omega_{\mathbb{P}^2}$. 
\item
If there is $\lL \in \Pic(X)$
with $\lL|_{D} \cong \oO_D(1)$, 
then there exist explicit polynomial relations
among stable pair invariants on $X$ and generalized DT
invariants (\ref{intro:rcm})
on $\omega_{\mathbb{P}^2}$ caused by 
$\mathrm{ST}_{\oO_{D}} \circ \otimes \lL$. 
\end{enumerate}
\end{thm}
The result of Theorem~\ref{thm:intro} (i) 
in particular derives a
recursion formula of 
stable pair invariants on 
$X$ with curve classes proportional to $[l]$
for a line $l \subset D$
(in other words
stable pair invariants on $\omega_{\mathbb{P}^2}$), 
whose coefficients involve 
the invariants (\ref{intro:rcm})
(cf.~Corollary~\ref{cor:recursion}). 
The result of Theorem~\ref{thm:intro} (ii) implies 
a stronger statement:
the stable pair invariants 
on $X$ with curve classes $\beta$ 
satisfying $D \cdot \beta<0$ are
described in terms of 
those with curve classes $\beta-c[l]$
for $c>0$, 
with coefficients involving (\ref{intro:rcm})
 (cf.~Remark~\ref{rmk:form}).
 
In the previous paper~\cite{TodS3}, 
the author proved a recursion formula of the 
generating series of 
the invariants (\ref{intro:rcm}) 
with $r>0$
in terms of theta type series for indefinite lattices. 
It is also possible 
to describe the invariants (\ref{intro:rcm})
with $r=0, c>0$
in terms of stable pair invariants on $X$ 
with curve classes proportional to $[l]$ (cf.~Lemma~\ref{rankzero}). 
These results imply that, in principle, 
one can compute the relations of stable pair 
invariants concerning 
Question~\ref{quest}
for the derived equivalences $\Phi$ and $\mathrm{ST}_{\oO_D}$. 
The resulting formulas in Theorem~\ref{thm:main}, 
Theorem~\ref{thm:const}
are complicated, and we leave it a future work to 
give a more conceptual understanding 
of our result. 

We should mention that the result of Theorem~\ref{thm:intro}
is still conditional to 
the following conjecture, 
which was also assumed in the author's previous work~\cite{Tcurve2}.
\begin{conj}\label{conj:critical}
Let $\mM$ be the moduli 
stack of objects 
$E \in D^b \Coh(X)$ satisfying 
$\Ext^{<0}(E, E)=0$. 
For $[E] \in \mM$, let 
$G$ be a maximal reductive 
subgroup of $\Aut(E)$.
Then there is a $G$-invariant analytic 
open neighborhood $V$ of $0$ in 
$\Ext^1(E, E)$, 
a $G$-invariant holomorphic function $f\colon V\to \mathbb{C}$
with $f(0)=df|_{0}=0$, and a smooth morphism 
of complex analytic stacks
\begin{align*}
\Phi \colon [\{df=0\}/G] \to \mM
\end{align*}
of relative dimension $\dim \Aut(E)- \dim G$. 
\end{conj}
The above conjecture
has been a technical obstruction 
to generalize
Joyce-Song's wall-crossing formula 
of DT invariants~\cite{JS} for 
coherent sheaves to the derived category.  
It was proved 
for $E \in \Coh(X)$ by Joyce-Song~\cite{JS}, 
and announced by Behrend-Getzler. 
There exist more recent progress 
toward it, which will be reviewed in 
the next subsection. 
Without assuming Conjecture~\ref{conj:critical}, 
we can prove Euler characteristic version
of Theorem~\ref{thm:intro}
(i.e. results 
for the naive Euler characteristics of stable pair 
moduli spaces), as 
stated in Subsection~\ref{subsec:Ever}.

\subsection{Related works}\label{subsec:related}
In~\cite{Tcurve2}, \cite{Cala}, 
the flop transformation formula 
of stable pair invariants was 
obtained 
from the categorical viewpoint, 
giving an answer to 
Question~\ref{quest} (i)
for birational Calabi-Yau 3-folds. 
In the orbifold case, let 
$Y$ be a Calabi-Yau 3-fold 
with Gorenstein quotient 
singularities and 
$X \to Y$
its crepant resolution. 
Under the HL
condition
on the associated 
Deligne-Mumford stack 
$\yY \to Y$, 
Bryan-Cadman-Young~\cite{BCY}
formulated a conjectural relationship 
between DT invariants on $X$ and those
on $\yY$. 
Combined with the DT/PT correspondence~\cite{BrH}, \cite{Tcurve1}, \cite{StTh}
on $X$, 
and Bayer's announced work~\cite{BaDTPT}
on it 
for CY3 orbifolds with HL 
 condition, 
we have a conjectural answer to Question~\ref{quest} (i)
in this situation. 
The conjecture in~\cite{BCY} is still open, but 
some progress toward it is obtained in~\cite{Cala2}, \cite{BrSt}, \cite{DR}.

In the above HL case, 
the resulting formula
should be
described by a product formula of the
generating series of stable pair invariants. 
In our situation of Theorem~\ref{thm:intro}, 
the stack $\yY$ does not satisfy the HL 
condition, and 
it seems unlikely that the 
results are formulated as product formulas of 
the generating series. 
From the categorical viewpoint, 
the main 
difference 
from the HL case
is the non-triviality 
of the Euler pairings between 
objects supported on 
the fibers of $X \to Y$. 
Due to this non-triviality, the 
combinatorics of the wall-crossing becomes complicated, 
and it seems hard to understand the result
in terms of the generating series. 
In any case, we hope
that the result of Theorem~\ref{thm:intro}
would give a hint toward a
generalization of the conjecture in~\cite{BCY}
without the HL condition. 

There exist few works
concerning Question~\ref{quest} (ii)
so far. 
We can say that the rationality of the generating 
series of stable pair invariants, 
conjectured in~\cite{PT} and proved in~\cite{Tolim2}, ~\cite{BrH},
is interpreted to be an answer to 
Question~\ref{quest} (ii) 
for the derived dualizing functor. 
Also the automorphic property of 
sheaf counting invariants on local K3 surfaces
under Hodge isometries, together with 
product expansion of the generating series of 
stable pair invariants on them~\cite{TodK3}
in terms of the former invariants,
is interpreted to be an answer to Question~\ref{quest} (ii)
for autoequivalences of K3 surfaces~\cite{Tst3}, \cite{TodK3}. 
The result of Theorem~\ref{thm:intro} (ii) 
provides a further example of 
such a phenomena. 

In GW theory, 
an analogue of Question~\ref{quest} (i) 
has been one of the central themes. 
Since 
birational Calabi-Yau 3-folds or orbifolds
should be
derived equivalent (cf.~\cite{Br1}, \cite{BKR}, \cite{Kawlog}), 
Question~\ref{quest} (i) for 
GW theory is related to the
analytic continuation problem of 
quantum cohomologies discussed 
in~\cite{Yong}, \cite{BrGr}, \cite{CIT}. 
Also 
we 
expect that Question~\ref{quest} (ii)
is related to the modularity problem
 of 
partition functions of GW invariants, 
as 
the action of 
autoequivalences on the derived category
should 
correspond to 
the monodromy action 
under the mirror symmetry. 
We refer to~\cite{OP}, \cite{MRS}
for the works on the 
modularity in GW theory.

In recent years, we have
seen progress toward 
an algebraic version of 
Conjecture~\ref{conj:critical}
using derived algebraic geometry. 
By the work of Pantev-To$\ddot{\textrm{e}}$n-Vaquie-Vezzosi~\cite{PTVV}, 
the stack $\mM$
is shown to be a derived stack 
with a $(-1)$-shifted
symplectic structure. Using this fact, 
B.~Bassat-Brav-Bussi-Joyce~\cite{BBBJ}
showed that $\mM$ has Zariski locally 
an atlas which is written as a critical 
locus of a certain algebraic function. 
Still this is not enough to conclude 
Conjecture~\ref{conj:critical}. 
However under the assumption 
that $\mM$ is Zariski locally 
written as a quotient stack of the form 
$[S/\GL_n(\mathbb{C})]$ 
for some complex scheme $S$, 
Bussi~\cite[Theorem~4.3]{Bussi}
showed 
a result which is very similar to 
Conjecture~\ref{conj:critical}.
Indeed her result implies
relevant Behrend function identities
for objects in $\mM$, 
which are enough for our applications. 
At this moment, the 
author does not know how to eliminate the 
local quotient stack assumption, nor prove it
in the situations we are interested in. 

\subsection{Ideas behind the proof of Theorem~\ref{thm:intro}}
\label{subsec:idea}
The proof of Theorem~\ref{thm:intro} follows from 
wall-crossing argument in the space of weak stability conditions, 
as in the author's previous 
papers~\cite{Tcurve1}, \cite{Tcurve2}, \cite{TodK3}. 
In order to explain the argument, we first
recall Bayer-Macri's description of 
the space $\Stab(\omega_{\mathbb{P}^2})$
of Bridgeland stability conditions on 
$D^b \Coh(\omega_{\mathbb{P}^2})$
in~\cite{BaMa}. 
They showed that 
the double quotient stack 
of $\Stab(\omega_{\mathbb{P}^2})$ by 
the actions of $\Aut D^b \Coh(\omega_{\mathbb{P}^2})$
and the additive group $\mathbb{C}$
contains the 
parameter space of the mirror family 
of $\omega_{\mathbb{P}^2}$. 
The latter space has three special 
points: large volume limit, conifold point
and orbifold point (cf. Figure~\ref{fig:one}). 
Near the large volume limit, the 
semistable objects
consist of (essentially) Gieseker semistable sheaves
on $\omega_{\mathbb{P}^2}$. 
At the orbifold point, 
the semistable objects
consist of representations of the McKay quiver
under derived McKay correspondence~\cite{BKR}. 
By taking a path connecting the
orbifold point with the 
large volume limit, one 
can relate representations of the McKay quiver
with semistable sheaves on $\omega_{\mathbb{P}^2}$
by wall-crossing phenomena: 
there is a finite number of walls on 
the above path such that 
the set of semistable objects are constant on 
the interval, but jump at walls. 

\begin{figure}[htbp]
 \begin{center}
  \includegraphics[width=60mm]{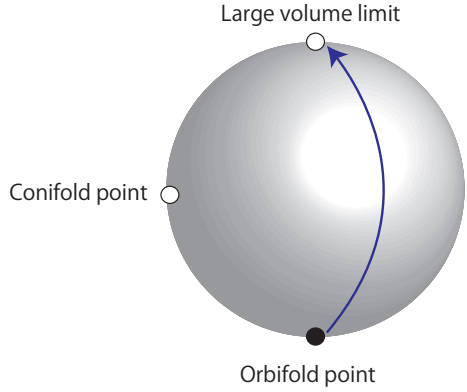}
 \end{center}
 \caption{The space of stability conditions on $\omega_{\mathbb{P}^2}$}
 \label{fig:one}
\end{figure}

Let us return to our global situation. 
In the situation of Theorem~\ref{thm:intro}, 
we define the following triangulated category
\begin{align*}
\dD_{X/Y} \cneq \langle \oO_X, D^b \Coh_{\le 1}(X/Y) \rangle_{\rm{tr}}. 
\end{align*}
Here $\Coh_{\le 1}(X/Y)$ 
is the category 
of coherent sheaves on $X$
which are at most one dimensional outside 
$D$. 
Our strategy 
is to construct a  
path similar to Figure~\ref{fig:one}
in the space of weak 
stability conditions on 
$\dD_{X/Y}$
\begin{align*}
\sigma_t \in \Stab_{\Gamma_{\bullet}}(\dD_{X/Y}), \ t \in \mathbb{R}_{\ge 0}. 
\end{align*}
The above one parameter family 
is an analogue of the path in Figure~\ref{fig:one}, i.e.
$\lim_{t\to \infty}\sigma_t$ corresponds to the large
volume limit, and $\sigma_0$
corresponds to the orbifold point. 
We show that the rank
one $\sigma_t$-stable objects 
for $t\gg 1$
consist of objects of the form 
\begin{align}\label{intro:OF}
\oO_X(rD) \otimes (\oO_X \to F)
\end{align}
for $r\in \mathbb{Z}$
and a stable pair 
$(\oO_X \to F)$
on $X$. 
We also show that the rank one 
$\sigma_0$-stable objects 
consist of objects of the form
\begin{align}\label{intro:OF2}
\Phi(\oO_{\yY} \to F)
\end{align}
for a stable 
pair $(\oO_{\yY} \to F)$ on $\yY$. 
Then we can relate objects (\ref{intro:OF}), (\ref{intro:OF2}) 
by wall-crossing phenomena. 
If we assume Conjecture~\ref{conj:critical}, 
then Joyce-Song's 
wall-crossing formula~\cite{JS}
is applied 
in our setting.
It
relates stable pair invariants on $\yY$ with those on $X$
together with the invariants (\ref{intro:rcm}), 
giving Theorem~\ref{thm:intro} (i). 

We now explain the idea of Theorem~\ref{thm:intro} (ii).
It follows from  
a general principle explained in~\cite[Section~1]{TGep}. 
In general, suppose that there is a stability 
condition $\tau$ on the derived category of 
a Calabi-Yau 3-fold, 
which has a symmetric property 
with respect to an  
autoequivalence $\Theta$
in a certain sense. 
In~\cite{TGep}, such a stability condition $\tau$
was called \textit{Gepner type} 
with respect to $\Theta$. 
Let
$\DT_{\tau}(v)$ 
be the DT type invariant (if it 
exists) counting 
$\tau$-semistable objects 
with numerical class $v$. 
The Gepner type property of $\tau$
would yield 
\begin{align}\label{DT:Gepner}
\DT_{\tau}(v)=\DT_{\tau}(\Theta_{\ast}v).
\end{align}
On the other hand, one may 
relate both sides of (\ref{DT:Gepner})
with 
classical DT invariants
counting sheaves or curves
by wall-crossing. 
Combined with the identity (\ref{DT:Gepner}), 
one may obtain 
non-trivial constraints among classical DT invariants
caused by $\Theta$. 

In Figure~\ref{fig:one}, the orbifold 
point is known to be 
Gepner type with respect to 
$\Theta=\mathrm{ST}_{\oO_{D}} \circ \otimes \lL$. 
Since the 
weak stability 
condition $\sigma_0$
on $\dD_{X/Y}$
is an analogue 
of the orbifold point, 
one 
expects that the above 
general philosophy
may be applied to obtain constraints 
among stable pair invariants on $X$
caused by $\Theta$.  
In our situation, the equivalence
$\Theta$
does not preserve $\dD_{X/Y}$, so $\sigma_0$
is not Gepner type in a strict sense. 
However one can prove that 
$\Theta$ takes $\sigma_0$-stable 
objects to similar stable objects 
in another triangulated category
\begin{align*}
\dD_{X/Y}^{\lL} \cneq \langle \lL, D^b \Coh_{\le 1}(X/Y) \rangle_{\rm{tr}}. 
\end{align*}
Namely there also exists a one parameter family
$\sigma_t^{\lL}$ of weak stability conditions on $\dD_{X/Y}^{\lL}$
such that $\sigma_0$-stable objects and $\sigma_0^{\lL}$-stable objects
coincide under the equivalence $\Theta$. 
We then apply the similar wall-crossing formula 
in $\dD_{X/Y}^{\lL}$ from $\sigma_0^{\lL}$ to $\sigma_t^{\lL}$
for $t \gg 1$. 
It implies another
description
of stable pair invariants on $\yY$ in terms of those on $X$ and 
the invariants (\ref{intro:rcm}). 
By comparing it with the result of Theorem~\ref{thm:intro} (i), 
we obtain the constraints in Theorem~\ref{thm:intro} (ii). 

\subsection{Plan of the paper}
In Section~\ref{sec:Derived}, we recall derived 
equivalences concerning Calabi-Yau 3-folds containing 
$\mathbb{P}^2$, and fix some notation. 
In Section~\ref{sec:Stable}, we recall stable pair invariants, 
generalized DT invariants on local $\mathbb{P}^2$, 
and their properties. 
In Section~\ref{sec:weak}, we
construct a one parameter family of weak 
stability conditions on the triangulated category $\dD_{X/Y}$. 
In Section~\ref{sec:Comp}, we describe the wall-crossing phenomena 
with respect to our weak stability conditions, 
and prove Theorem~\ref{thm:intro}. 

\subsection{Acknowledgment}
The author would like to thank Arend Bayer for
explaining his announced work~\cite{BaDTPT}. 
This work is supported by World Premier 
International Research Center Initiative
(WPI initiative), MEXT, Japan. This work is also supported by Grant-in Aid
for Scientific Research grant (No. 26287002)
from the Ministry of Education, Culture,
Sports, Science and Technology, Japan.

\subsection{Notation and convention}
In this paper, all 
the varieties or stacks 
are defined over $\mathbb{C}$. 
For a $d$-dimensional variety $X$, we denote by $H^{\ast}(X, \mathbb{Q})$ 
the even part of the singular cohomologies of $X$, and
write its element as 
$(v^0, v^1, \cdots, v^d)$ for $v^i \in H^{2i}(X, \mathbb{Q})$.
We sometimes abbreviate $\mathbb{Q}$ and just write 
$H^{2i}(X, \mathbb{Q})$, $H_{2i}(X, \mathbb{Q})$
as $H^{2i}(X)$, $H_{2i}(X)$. 
For a triangulated category $\dD$ and a set of objects
$S$ in $\dD$, we denote by 
$\langle S \rangle_{\rm{tr}}$ the triangulated closure, i.e. 
the smallest triangulated category of $\dD$ 
which contains $S$. 
Also 
$\langle S \rangle_{\rm{ex}}$ is the extension closure, i.e. 
the smallest extension closed subcategory in $\dD$
which contains $S$.
For a 
Deligne-Mumford stack $\yY$, we denote by 
$\Coh(\yY)$
the abelian category of coherent sheaves on $\yY$. 
For $d\in \mathbb{Z}$, we denote by 
$\Coh_{\le d}(\yY) \subset \Coh(\yY)$
the subcategory of objects $E\in \Coh(\yY)$ 
satisfying $\dim \Supp(E) \le d$. 
If $d=0$, we write the subscript `$\le 0$' just as `$0$', 
e.g. write $\Coh_{\le 0}(X)$ as $\Coh_0(X)$, etc.  
For $E \in D^b \Coh(\yY)$, we denote by 
$\hH^i(E) \in \Coh(\yY)$ its $i$-th 
cohomology. 

\section{Derived category of Calabi-Yau 3-folds containing $\mathbb{P}^2$}
\label{sec:Derived}
\subsection{Geometry of Calabi-Yau 3-folds containing $\mathbb{P}^2$}
Let $X$ be a smooth projective Calabi-Yau 3-fold, i.e. 
\begin{align*}
K_X=0, \ H^1(X, \oO_X)=0. 
\end{align*}
We always assume that 
there is a closed embedding 
\begin{align*}
i \colon \mathbb{P}^2 \hookrightarrow X
\end{align*}
whose image we denote by $D$. 
There exist several examples of such 
Calabi-Yau 3-folds, as follows:  
\begin{exam}\label{exam1}
Let $X$ be the hypersurface in $\mathbb{P}^3 \times \mathbb{P}^1$
given by
\begin{align*}
X=\left\{ y_1 \left( \sum_{i=1}^4 x_i^4 \right)
 + y_2 \prod_{i=1}^4 x_i=0 \right\}. 
\end{align*}
Here $x_i, 1\le i\le 4$ are homogeneous coordinates of $\mathbb{P}^3$
and $y_i, 1\le i \le 2$ are those of $\mathbb{P}^1$. 
Then $X$ is a smooth Calabi-Yau 3-fold 
which contains planes $(y_1=x_i=0)$
for each $1\le i\le 4$. 
\end{exam}
\begin{exam}\label{exam2}
Let $S_1 \subset \mathbb{P}^3$ be a plane, 
$S_7 \subset \mathbb{P}^3$ a smooth 
hypersurface of degree seven, 
such that the divisor $S_1+S_7$ is 
a normal crossing divisor. 
Let $Z \to \mathbb{P}^3$ be the double 
cover branched along $S_1 + S_7$. Then $Z$ 
has $A_1$-singularities along the pull-back of $S_1 \cap S_7$. 
Let 
\begin{align*}
X \to Z
\end{align*} be the blow-up along the singular locus. 
Then $X$ is a smooth Calabi-Yau 3-fold, and 
the pull-back of $S_1$ contains $\mathbb{P}^2$ as an irreducible component. 
See~\cite[Section~4.2]{Kapu}.
\end{exam}

\begin{exam}\label{exam3}
 Let $E$ be the elliptic curve which admits
an automorphism of order $3$. The group 
$G=\mathbb{Z}/3\mathbb{Z}$ acts on $E \times E \times E$
diagonally. Then the 
crepant resolution
\begin{align*}
X \to (E \times E \times E)/G
\end{align*} 
is a smooth Calabi-Yau 3-fold which 
contains 27 planes. 
See~\cite[Section~2]{Bea}. 
\end{exam}
Let $H$ be an ample divisor on $X$ and 
$l \subset D$ a line. 
We note that
\begin{align*}
D \cdot l=-3, \ D^2 =-3[l], \ D^3=9. 
\end{align*}
The divisor
$3H +(H \cdot l) D$
is nef and big on $X$. By the basepoint 
free theorem, 
some multiple of it
gives a birational morphism
\begin{align}\label{fXY}
f \colon X \to Y
\end{align}
which contacts a divisor $D \subset X$ to a point $p \in Y$. 
It is well-known that\footnote{For example, the same argument 
of~\cite[(3.3.5)]{Mori}
shows the isomorphism (\ref{OYp}). }
\begin{align}\label{OYp}
\widehat{\oO}_{Y, p} 
\cong \mathbb{C}\db[ x, y, z \db]^{G}.
\end{align}
Here $G \cneq \mathbb{Z}/3\mathbb{Z}$ acts on $\mathbb{C}\db[x, y, z\db]$ via 
weight $(1, 1, 1)$. Since $p \in Y$
is a quotient singularity, we have the 
associated smooth Deligne-Mumford stack\footnote{We 
refer to~\cite[Definition~2.1]{Kawlog} for the construction of $\yY$. 
We note that $Y$ satisfies the condition $(\ast)$
in~\cite[Definition~2.1]{Kawlog} since $p\in Y$ is a quotient singularity.}
\begin{align}\label{yYY}
g \colon \yY \to Y
\end{align}
whose coarse moduli space is isomorphic to $Y$. 
The diagram
\begin{align}\label{dia:global}
\xymatrix{
&  \ar[dl]_{p} X \times_Y \yY \ar[dr]^{q} & \\
X \ar[dr]_f \ar@{.>}[rr]^{\phi} &   &  \yY \ar[dl]^g \\
&  Y  &
}
\end{align}
is pulled back 
via $\Spec \widehat{\oO}_{Y, p} \to Y$ to 
the pull-back via $\Spec \widehat{\oO}_{\mathbb{C}^3/G, 0} \to \mathbb{C}^3/G$
of the standard McKay 
diagram\footnote{By abuse of notation, 
in the diagram (\ref{dia:loc}), we use the same notation 
for the morphisms used in the diagram (\ref{dia:global}).}
of local $\mathbb{P}^2$
\begin{align}\label{dia:loc}
\xymatrix{
& \ar[dl]_{p}  [\mathrm{Bl}_0 \mathbb{C}^3/G] \ar[dr]^{q} \\
U=\omega_{\mathbb{P}^2}
 \ar[dr]_f \ar@{.>}[rr]^{\phi} &   &  [\mathbb{C}^3/G] \ar[dl]^g \\
&  \mathbb{C}^3/G.
  &
}
\end{align}
 Here $\mathrm{Bl}_0 \mathbb{C}^3 \to \mathbb{C}^3$ is the blow-up 
at the origin, which admits the $G$-action since $G$ 
fixes the origin. Also 
\begin{align}\label{pi}
\pi \colon 
U=\omega_{\mathbb{P}^2} \to \mathbb{P}^2
\end{align}
is the total space of the 
canonical line bundle of $\mathbb{P}^2$, which 
is the coarse moduli space of the quotient stack 
$[\mathrm{Bl}_0 \mathbb{C}^3/G]$.

\subsection{Derived equivalence}
Since the diagram (\ref{dia:loc})
is toric, we can apply~\cite[Theorem~4.2]{Kawlog}
(also see~\cite{BKR})
to show the derived equivalence
\begin{align}\label{Phi}
\Phi \cneq 
\dR p_{\ast} \dL q^{\ast} \colon D^b \Coh(\yY) \stackrel{\sim}{\to}
D^b \Coh(X). 
\end{align}
Let $\rho_j$ be the 
one dimensional representation of $G$ with weight $j$. 
The objects
\begin{align*}
S_j \cneq \rho_{-j} \otimes \oO_0 \in \Coh_G(\mathbb{C}^3)
\end{align*}
naturally define the objects
$S_j \in \Coh(\yY)$. 
Here $\Coh_G(\mathbb{C}^3)$ is the 
category of $G$-equivariant coherent sheaves on $\mathbb{C}^3$, 
which coincides with the category of coherent sheaves on $[\mathbb{C}^3/G]$. 
Also we set
\begin{align}\label{def:T}
T_0 \cneq \oO_{D}, \ 
T_1 \cneq \Omega_{D}(1)[1], \ 
T_2 \cneq \oO_{D}(-1)[2]. 
\end{align}
We have the following lemma. 
\begin{lem}\label{lem:ST}
We
have $\Phi(\oO_{\yY}) \cong \oO_X$
and $\Phi(S_j)\cong T_j$
for $0\le j\le 2$.  
\end{lem}
\begin{proof}
It is enough to prove the same claim for the local
derived equivalence
\begin{align}\notag
\Phi=\dR p_{\ast} \dL q^{\ast} \colon D^b \Coh_G(\mathbb{C}^3) 
\stackrel{\sim}{\to}
D^b \Coh(\omega_{\mathbb{P}^2}).
\end{align}
Let $\eE$ be the vector bundle on $U=\omega_{\mathbb{P}^2}$ given by
\begin{align*}
\eE=\oO_U \oplus \oO_U(1) \oplus \oO_U(2). 
\end{align*}
It is well-known that 
we have the derived equivalence
(cf.~\cite{Brs5})
\begin{align}\label{E}
\dR \Hom (\eE, -) \colon 
D^b \Coh(U) \stackrel{\sim}{\to} D^b \modu (B)
\end{align}
where $B$ is the non-commutative algebra 
defined by $\End(\eE)$. 
The algebra $B$ is the path algebra of 
a quiver with three vertices and some relations.
Under the equivalence (\ref{E}), the objects
$T_0$, $T_1$, $T_2$
are sent to the simple objects corresponding to 
the above three vertices. 
Therefore the isomorphism $\Phi(S_j) \cong T_j$ follows
if we show that
\begin{align}\label{Rhom}
\dR \Hom(\oO_U(k), \Phi(S_j))=\mathbb{C}^{\delta_{jk}}
\end{align}
for $0\le j, k \le 2$. 
Let $p' \colon \mathrm{Bl}_0(\mathbb{C}^3) \to U$ be the 
Galois $G$-cover. Then we have
\begin{align*}
p'_{\ast} \oO_{\mathrm{Bl}_0(\mathbb{C}^3)}
\cong \oO_U \oplus \oO_U(1) \otimes \rho_1 \oplus \oO_U(2) \otimes \rho_2. 
\end{align*}
It follows that 
\begin{align}\notag
\Phi(\oO_{\mathbb{C}^3} \otimes \rho_{-k})
&\cong (\oO_U \otimes \rho_{-k} \oplus \oO_U(1) \otimes \rho_{1-k} 
\oplus \oO_U(2) \otimes \rho_{2-k})^{G} \\ \label{cong:OU}
&\cong \oO_U(k)
\end{align}
 for $0\le k\le 2$.
Therefore (\ref{Rhom}) holds.
The isomorphism $\Phi(\oO_{\yY}) \cong \oO_X$
also follows from the isomorphism (\ref{cong:OU})
for $k=0$. 
\end{proof}

We define the following abelian category
\begin{align}\label{Coh1}
\Coh_{\le d}(X/Y) \cneq \{ E \in \Coh(X) : \dim \Supp f_{\ast}E \le d\}. 
\end{align}
By the construction of $\Phi$, we have the commutative diagram
\begin{align}\label{cohd}
\xymatrix{
D^b \Coh_0(\yY) \ar[d]^{\Phi} \ar@{^{(}->}[r] & D^b \Coh_{\le 1}(\yY)
 \ar[d]^{\Phi}
 \ar@{^{(}->}[r] & D^b \Coh(\yY) \ar[d]^{\Phi} \\
D^b \Coh_0(X/Y) \ar@{^{(}->}[r] &  D^b \Coh_{\le 1}(X/Y)
\ar@{^{(}->}[r] 
& D^b \Coh(X)
}
\end{align}
such that each vertical 
arrows are equivalences.

\subsection{Numerical Grothendieck groups}
Let $\dD$ be a $\mathbb{C}$-linear triangulated 
category satisfying
\begin{align*}
\sum_{i\in \mathbb{Z}} \dim \Hom(E, F[i])<\infty
\end{align*}
for all $E, F \in \dD$. 
Under the above condition, the 
following Euler pairing is well-defined
\begin{align}\label{Euler}
\chi(E, F) \cneq \sum_{i\in \mathbb{Z}}(-1)^i \dim \Hom(E, F[i]). 
\end{align}
The above pairing descends to the pairing on 
the Grothendick group $K(\dD)$ of $\dD$. 
Two elements $E_1, E_2 \in K(\dD)$ are called \textit{numerically 
equivalent} if we have
$\chi(E_1, F)=\chi(E_2, F)$ for any $F \in K(\dD)$. 
The \textit{numerical Grothendieck group} $N(\dD)$
of $\dD$ is defined to be
the group of numerical equivalence classes of $K(\dD)$. 
By the definition, the Euler pairing (\ref{Euler})
descends to the 
perfect pairing on $N(\dD)$.

Let $X$, $\yY$ be as in the previous subsections. 
We set
\begin{align*}
N(X) \cneq N(D^b \Coh(X)), \ N(\yY) \cneq N(D^b \Coh(\yY)). 
\end{align*}
The equivalence (\ref{Phi}) 
induces the isomorphism
\begin{align}\label{PhiN}
\Phi_{\ast} \colon N(\yY) \stackrel{\sim}{\to} N(X). 
\end{align}
Since $X$ is a smooth projective 
3-fold, it satisfies the Hodge conjecture. 
Together with Riemann-Roch theorem, 
the Chern character map from $K(X)$ 
descends to the injective homomorphism
\begin{align}\label{ch:inj}
\ch \colon N(X) \hookrightarrow H^{\ast}(X, \mathbb{Q}). 
\end{align}
In particular, both of $N(X)$ and $N(\yY)$ are 
finitely generated free abelian groups. 
The Euler paring (\ref{Euler}) on $N(X)$ is described as 
\begin{align}\notag
\chi(E, F)=&\sum_{j=0}^{3} (-1)^{j}\ch_{j}(E) \ch_{3-j}(F) \\
\label{Euler2}
&+\frac{c_2(X)}{12} \left(\ch_0(E)\ch_1(F)-\ch_1(E) \ch_0(F) \right).
\end{align}
We set
\begin{align*}
&N_{\le d}(\yY) \cneq \Imm (K(\Coh_{\le d}(\yY)) \to N(\yY)) \\
&N_{\le d}(X/Y) \cneq \Imm (K(\Coh_{\le d}(X/Y)) \to N(X)). 
\end{align*}
By the diagram (\ref{cohd}), we have the commutative diagram
\begin{align*}
\xymatrix{
\bigoplus_{j=0}^{2} \mathbb{Z}[S_j] \ar[d]_{\cong}^{[S_j] \mapsto [T_j]} 
\ar@{^{(}->}[r]^{\cong} &
N_0(\yY) \ar[d]_{\cong}^{\Phi_{\ast}}
 \ar@{^{(}->}[r] & N_{\le 1}(\yY) \ar[d]_{\cong}^{\Phi_{\ast}}
 \ar@{^{(}->}[r] & N(\yY) \ar[d]_{\cong}^{\Phi_{\ast}} \\
K(\mathbb{P}^2) \ar[r]^{i_{\ast}}_{\cong} \ar[d]^{\ch} &
N_0(X/Y) \ar@{^{(}->}[r] &  N_{\le 1}(X/Y) \ar[d]^{\ch} \ar@{^{(}->}[r] 
& N(X) \ar[d]^{\ch} \\
H^{\ast}(\mathbb{P}^2, \mathbb{Q}) \ar[rr]^{i_{\sharp}} &  &  
\mathbb{Q}[D] \oplus 
H^{\ge 4}(X, \mathbb{Q}) \ar@{^{(}->}[r] & H^{\ast}(X, \mathbb{Q}). 
}
\end{align*}
\begin{rmk}\label{rmk:K}
It is well-known that
$K(\mathbb{P}^2) \cong \mathbb{Z}^3$, and 
the Euler paring on it is perfect. 
This fact easily shows that the map
$i_{\ast} \colon K(\mathbb{P}^2) \to N_0(X/Y)$
is an isomorphism. 
\end{rmk}
In what follows, we fix the 
isomorphism
\begin{align}\label{isom:P}
\mathbb{Q}^{\oplus 3} \stackrel{\cong}{\to}
H^{\ast}(\mathbb{P}^2, \mathbb{Q}), \ 
(r, c, m) \mapsto (r, c\cdot h, m \cdot h^2)
\end{align}
where $h=c_1(\oO_{\mathbb{P}^2}(1))$, and write 
elements of $H^{\ast}(\mathbb{P}^2, \mathbb{Q})$
as $(r, c, m) \in \mathbb{Q}^3$ via (\ref{isom:P}). 
In this notation, the map $i_{\sharp}$ is given by
\begin{align}\label{iflat}
i_{\sharp}(r, c, m)=\left(0, r[D], \left(\frac{3}{2}r + c\right)[l], 
\frac{3}{2}r+ \frac{3}{2}c
+ m \right)
\end{align}
by the Grothendieck Riemann-Roch theorem.

\subsection{Seidel-Thomas twist}\label{subsec:ST}
The object $\oO_D \in \Coh(X)$ is a spherical object, i.e.
\begin{align*}
\Ext_{X}^{i}(\oO_D, \oO_D)=\left\{ \begin{array}{cl}
\mathbb{C}  &  i=0, 3 \\
0 & i \neq 0, 3. 
\end{array}  \right. 
\end{align*}
By~\cite{ST}, there is the associated autoequivalence
\begin{align*}
\mathrm{ST}_{\oO_D} \colon D^b \Coh(X) \stackrel{\sim}{\to} D^b \Coh(X)
\end{align*}
called \textit{Seidel-Thomas twist}. 
For any $E \in D^b \Coh(X)$, the 
object
$\mathrm{ST}_{\oO_D}(E)$ fits into 
the distinguished triangle
\begin{align}\label{ST}
\dR \Hom(\oO_D, E) \otimes \oO_D \to E \to \mathrm{ST}_{\oO_D}(E). 
\end{align}
We now assume that\footnote{It is easy to see that such a line bundle exists 
in Examples~\ref{exam1}, \ref{exam2}, \ref{exam3}.} 
there is a line 
bundle $\lL$ on $X$ satisfying
$i^{\ast}\lL \cong \oO_{\mathbb{P}^2}(1)$. 
The isomorphism (\ref{cong:OU}) for $k=1$
shows that 
the object
\begin{align}\label{Ldag}
\lL^{\dag} \cneq \Phi^{-1}(\lL)
\end{align}
is also a line bundle on $\yY$. 
We set
\begin{align}\label{Theta}
\Theta \cneq \mathrm{ST}_{\oO_D} \circ \otimes \lL \colon 
D^b \Coh(X) \stackrel{\sim}{\to} D^b \Coh(X).  
\end{align}
\begin{lem}\label{lem:Ldag}
The equivalence
\begin{align}\label{Theta+}
\Theta^{\dag} \cneq 
\Phi^{-1} \circ \Theta \circ \Phi \colon 
D^b \Coh(\yY) \to D^b \Coh(\yY)
\end{align}
restricts to the autoequivalence of $\Coh(\yY)$
given by $-_{\otimes \oO_{\yY}}\lL^{\dag}$.  
\end{lem}
\begin{proof}
We have the local autoequivalence
\begin{align}\label{Theta:loc}
\Theta^{\dag}
\colon D^b \Coh_G(\mathbb{C}^3)
\to D^b \Coh_G(\mathbb{C}^3)
\end{align}
constructed in the same way of (\ref{Theta+}), 
replacing 
$\lL$ by $\oO_U(1)$. 
By the isomorphism (\ref{cong:OU}) for $k=1$, 
it is enough to prove 
that the functor (\ref{Theta:loc})
is isomorphic to tensoring  
$\oO_{\mathbb{C}^3}\otimes \rho_{-1}$. 
A direct computation easily shows that
\begin{align*}
\Theta(\oO_U)\cong \oO_U(1), \
\Theta(\oO_U(1)) \cong \oO_U(2), \ 
\Theta(\oO_U(2)) \cong \oO_U. 
\end{align*}
Here $\Theta=\mathrm{ST}_{\oO_{\mathbb{P}^2}} \circ \oO_U(1)$
by abuse of notation. 
This implies that 
\begin{align*}
\Theta^{\dag}(\oO_{\mathbb{C}^3} \otimes \rho_{-j})
\cong \oO_{\mathbb{C}^3} \otimes \rho_{-j-1}
\end{align*}
 for $j\in \mathbb{Z}/3\mathbb{Z}$. 
Hence the functor (\ref{Theta:loc}) is isomorphic to 
tensoring $\oO_{\mathbb{C}^3}\otimes \rho_{-1}$
on the objects $\oO_{\mathbb{C}^3} \otimes \rho_{-j}$. 
Since the objects $\oO_{\mathbb{C}^3} \otimes \rho_{-j}$
for $0\le j\le 2$ 
are local projective generators of 
$\Coh_G(\mathbb{C}^3)$, we obtain the result. 
\end{proof}

Let $\Coh_{\le d}(X/Y)$ be the 
abelian category defined by (\ref{Coh1}). 
By Lemma~\ref{lem:Ldag}, we have the commutative diagram
\begin{align}\label{cohd2}
\xymatrix{
\Coh_{\le d}(\yY) \ar[d]^{\otimes \lL^{\dag}} \ar@{^{(}->}[r] & D^b \Coh_{\le d}(\yY)
 \ar[d]^{\otimes \lL^{\dag}}
 \ar[r]^{\Phi} & D^b \Coh_{\le d}(X/Y) \ar[d]^{\Theta} \\
\Coh_{\le d}(\yY) \ar@{^{(}->}[r] &  D^b \Coh_{\le d}(\yY)
\ar[r]^{\Phi} 
& D^b \Coh_{\le d}(X/Y)
}
\end{align}
such that each vertical arrows are equivalences. 
The equivalence (\ref{Theta}) also induces the 
commutative diagram
\begin{align}\label{cohd3}
\xymatrix{
N_{\le 1}(X/Y) \ar[r]^{\Theta_{\ast}} \ar[d]_{\ch}
 & N_{\le 1}(X/Y) \ar[d]^{\ch} \\
\mathbb{Q}[D] \oplus 
H^{\ge 4}(X, \mathbb{Q}) \ar[r]^{\Theta_{\sharp}} &
\mathbb{Q}[D] \oplus H^{\ge 4}(X, \mathbb{Q}). 
}
\end{align}
Using the triangle
 (\ref{ST}), the linear isomorphism $\Theta_{\sharp}$
is calculated as   
\begin{align}\notag
\Theta_{\sharp}(r[D], \beta, n)
=&\left(\left(\frac{5}{2}r+\beta D \right)[D], 
\beta+\left(\frac{13}{4}r+\frac{3}{2}\beta D \right)[l], \right. \\ 
\label{dia:Theta}
&\hspace{40mm}
\left. n +\frac{11}{4}r +\frac{3}{2}\beta D + c_1(\lL) \beta \right). 
\end{align}

\section{Stable pairs and generalized DT invariants}\label{sec:Stable}
In this section, we recall stable pairs, 
generalized DT invariants, and their properties. 
\subsection{Stable pair invariants}
First we recall the definition of 
stable pairs on a Calabi-Yau 3-fold $X$
introduced by Pandharipande-Thomas~\cite{PT}. 
\begin{defi}\emph{(\cite{PT})}
A stable pair on $X$
consists of data
\begin{align}\label{SP}
s\colon \oO_X \to F
\end{align}
where $F \in \Coh_{\le 1}(X)$ is pure one dimensional, 
and $\Cok(s) \in \Coh_0(X)$. 
\end{defi}

For $n \in \mathbb{Z}$ and $\beta \in H_2(X, \mathbb{Z})$, 
let
\begin{align*}
P_n(X, \beta)
\end{align*}
be
the moduli space of stable pairs
(\ref{SP}) with $[F]=\beta$\footnote{If we write 
$[F]=\beta$, it means that 
the fundamental homology class determined by $F$
equals to $\beta$. By abuse of notation, 
we also use the notation $[F]$ for the
class of $F$ in the numerical Grothendieck group.}
 and $\chi(F)=n$. 
The moduli space $P_n(X, \beta)$ is 
a projective scheme with a symmetric perfect
obstruction theory. 
It is also
regarded as a moduli space of two term complexes 
$(\oO_X \stackrel{s}{\to} F)$, 
satisfying the condition
\begin{align*}
\ch(\oO_X \to F)=(1, 0, -\beta, -n). 
\end{align*}
Here $\oO_X$ is located in degree zero, 
and $H_2(X)$ is 
identified with $H^4(X)$ by Poincar\'e duality. 
Let $\nu$ be the Behrend constructible function~\cite{Beh}
on $P_n(X, \beta)$. 
The stable pair invariant $P_{n, \beta}(X)$ is defined by
\begin{align*}
P_{n, \beta}(X) \cneq \int_{P_n(X, \beta)} \nu \ d\chi.  
\end{align*}
Here for a constructible function $\nu$ on a variety $M$, we
define
\begin{align*}
\int_{M} \nu \ d\chi \cneq \sum_{m\in \mathbb{Z}}
m \cdot \chi(\nu^{-1}(m)). 
\end{align*}

\begin{rmk}\label{rmk:Ng}
By the injectivity of (\ref{ch:inj}), 
the map
\begin{align}\label{NH2}
N_{\le 1}(X) \to H_2(X, \mathbb{Z}) \oplus \mathbb{Z} 
\end{align}
sending $F$ to $([F], \chi(F))$ is injective.
Therefore if $P_n(X, \beta) \neq \emptyset$, 
we may write it as $P(X, \alpha)$, and 
$P_{n, \beta}(X)$ as $P_{\alpha}(X)$, for $\alpha \in N_{\le 1}(X)$
corresponding to $(\beta, n)$ under (\ref{NH2}). 
\end{rmk}

\begin{rmk}\label{rmk:supportD}
Suppose that $X$ contains a divisor $D \cong \mathbb{P}^2$, 
$\beta=c[l]$ for $c>0$ and a line $l\subset D$. 
Then then for any stable pair $(\oO_X \to F) \in P_n(X, c[l])$, 
the sheaf $F$ is supported on $D$. 
Since the formal neighborhood of $D \subset X$
and the zero section 
$\mathbb{P}^2 \subset \omega_{\mathbb{P}^2}$
are isomorphic, the invariant 
$P_{n, c[l]}(X)$ coincides with the stable pair invariant on 
$\omega_{\mathbb{P}^2}$. 
\end{rmk}

It is straightforward to generalize the notion of stable pairs 
to a CY3 orbifold $\yY$. 
\begin{defi}
An orbifold stable pair on $\yY$
consists of data
\begin{align}\label{STY}
s \colon \oO_{\yY} \to F
\end{align}
where $F \in \Coh_{\le 1}(\yY)$ is pure, 
i.e. $\Hom(\Coh_0(\yY), F)=0$, 
and $\Cok(s) \in \Coh_0(\yY)$. 
\end{defi}
Similarly to Remark~\ref{rmk:Ng},
we denote by 
\begin{align*}
P(\yY, \gamma)
\end{align*}
for $\gamma \in N_{\le 1}(\yY)$
the 
moduli space of orbifold stable pairs (\ref{STY})
 satisfying 
$[F]=\gamma$ in $N_{\le 1}(\yY)$. 
In this paper, we don't purse 
the foundation of orbifold stable pair moduli spaces, 
e.g. GIT construction of $P(\yY, \gamma)$. 
At least we have the following lemma, 
which is enough for our purpose. 
\begin{lem}
The moduli space $P(\yY, \gamma)$ is 
a finite type open sub algebraic space of 
the moduli space of simple objects in $D^b \Coh(\yY)$. 
\end{lem}
\begin{proof}
The result is an immediate consequence of 
Proposition~\ref{prop:fund} and Proposition~\ref{prop:PY} proven later. 
\end{proof}
Using the Behrend function $\nu$ on $P(\yY, \gamma)$, 
the orbifold stable pair invariant is defined by 
\begin{align}\label{PgY}
P_{\gamma}(\yY) \cneq \int_{P(\yY, \gamma)} \nu \ d\chi. 
\end{align}
Note that $P_{\gamma}(\yY)=0$ for $0\neq \gamma \in N_0(\yY)$
by the definition of orbifold stable pairs. 

\subsection{Generalized DT invariants}\label{subsec:genDT}
We recall the 
construction of generalized DT invariants on a 
Calabi-Yau 3-fold $X$, following~\cite{JS}. 
Let $\cC oh(X)$ be the moduli 
stack of all the objects in $\Coh(X)$. 
The stack theoretic Hall algebra $H(X)$ is 
$\mathbb{Q}$-spanned by the isomorphism 
classes of the symbols 
\begin{align}\notag
[\rho \colon \xX \to \cC oh(X)]
\end{align}
where $\xX$ is an Artin stack of finite type over
$\mathbb{C}$ with 
affine geometric stabilizers and
$\rho$ is a 1-morphism. 
The relation is generated by 
\begin{align}\label{Hall:rel}
[\rho \colon \xX \to \cC oh(X)]
\sim [\rho|_{\yY} \colon \yY \to \cC oh(X)]
+ [\rho|_{\uU} \colon \uU \to \cC oh(X)]
\end{align}
where $\yY \subset \xX$ is a closed substack and 
$\uU\cneq \xX \setminus \yY$. 
There is an associative $\ast$-product
on $H(X)$
based on the Ringel-Hall algebras (cf.~\cite[Section~5.1]{Joy2}).
The unit is given by 
\begin{align*}
1=[\Spec \mathbb{C} \to \cC oh(X)]
\end{align*}
which corresponds to $0\in \Coh(X)$. 
Also there is a Lie subalgebra 
\begin{align*}
H^{\rm{Lie}}(X) \subset H(X)
\end{align*}
consisting of elements 
supported on \textit{virtual indecomposable 
objects}. 
We refer to~\cite[Section~5.2]{Joy2} for the detail
of the definition of $H^{\rm{Lie}}(X)$. 

Let $C(X)$ be the Lie algebra
\begin{align*}
C(X) \cneq \bigoplus_{v \in N(X)}
\mathbb{Q} \cdot c_{v}
\end{align*}
with bracket given by 
\begin{align}\label{bracket}
[c_{v_1}, c_{v_2}] \cneq (-1)^{\chi(v_1, v_2)}
\chi(v_1, v_2) c_{v_1+v_2}. 
\end{align}
 By~\cite[Theorem~5.12]{JS}, there is a Lie algebra 
homomorphism
\begin{align}\label{PiHall}
\Pi \colon H^{\rm{Lie}}(X) \to C(X)
\end{align}
such that if $\xX$ is a $\mathbb{C}^{\ast}$-gerbe over an 
algebraic space $\xX'$, we have 
\begin{align}\notag
\Pi([\rho \colon \xX \to \cC oh_{v}(X)])
=-\left(\sum_{k\in \mathbb{Z}}
k \cdot \chi(\nu^{-1}(k))  \right) c_{v}.
\end{align}
Here $\cC oh_v(X)$ is the stack of sheaves 
with numerical class $v$, $\rho$ is an open immersion
and 
$\nu$ is  
Behrend's constructible function
on $\xX'$. 

Let $H$ be an ample divisor on $X$. 
For $v \in N(X)$, 
let\footnote{We omit $H$ in the 
subscript of the moduli spaces and invariants, 
as its choice does not matter in the situation 
of Subsection~\ref{subsec:DTP2}.} 
\begin{align}\label{stack}
\mM^{s(ss)}(v) \subset \cC oh(X)
\end{align}
be the open substack of
$H$-Gieseker (semi)stable 
sheaves $E \in \Coh(X)$ satisfying
$[E]=v$. 
The stack (\ref{stack}) determines the element
\begin{align}\notag
\delta(v) \cneq [\mM^{ss}(v) \subset 
\cC oh(X)] \in H(X). 
\end{align}
The above element also defines the element of $H^{\rm{Lie}}(X)$: 
\begin{align}\notag
\epsilon(v)
\cneq \sum_{\begin{subarray}{c}
k\ge 1, v_1+ \cdots + v_k=v \\
p(v_i)=p(v)
\end{subarray}}
\frac{(-1)^{k-1}}{k}
\delta(v_1) \ast \cdots \ast \delta(v_k).
\end{align}
Here 
$p(v)$ is the reduced Hilbert polynomial 
of a sheaf $E$ with $[E]=v$. 
\begin{defi}
The generalized DT invariant
$\DT(v) \in \mathbb{Q}$
is defined
by the formula:
\begin{align}\notag
\Pi(\epsilon(v))=-\DT(v) \cdot c_{v}. 
\end{align}
\end{defi}
\begin{rmk}\label{rmk:s=ss}
If $\mM^{s}(v)=\mM^{ss}(v)$, 
then they are $\mathbb{C}^{\ast}$-gerbe over a
quasi-projective scheme $M^{s}(v)$. 
In this case, the invariant $\DT(v)$
is written as
\begin{align*}
\DT(v)=\int_{M^{s}(v)} \nu \ 
d\chi
\end{align*}
where $\nu$ is the Behrend function on $M^{s}(v)$. 
\end{rmk}
It has been expected that 
the above arguments
are generalized to the derived category
setting. 
Namely, 
we expect that we can replace
the category of coherent sheaves by 
other heart of t-structure $\aA \subset D^b \Coh(X)$, 
and Gieseker stability by Bridgeland stability~\cite{Brs1}
or weak stability~\cite{Tcurve1}. 
The arguments are almost parallel, except 
that we need 
one technical result on the 
local description of the moduli stack of 
objects in the derived category, 
proven for coherent sheaves in~\cite[Theorem~5.3]{JS}. 
Conjecture~\ref{conj:critical} in
the introduction is required to show
the existence of a Lie algebra homomorphism (\ref{PiHall})
in the derived category setting.

\subsection{Generalized DT invariants on local $\mathbb{P}^2$}
\label{subsec:DTP2}
The construction of generalized 
DT invariants 
is also applied to the non-compact
Calabi-Yau 3-fold 
$U=\omega_{\mathbb{P}^2}$.
Let 
$\Coh_{\rm{c}}(U)$
be the category of 
coherent sheaves on $U$ with 
compact supports, 
and set
\begin{align*}
\Lambda \cneq \Imm (\ch \colon \Coh(\mathbb{P}^2) \setminus \{0\}
\to H^{\ast}(\mathbb{P}^2, \mathbb{Q})). 
\end{align*}
Under the isomorphism (\ref{isom:P}), 
 we have
\begin{align*}
\Lambda \subset \left\{ (r, c, m) \in \mathbb{Z}^{2}
\oplus \frac{1}{2}\mathbb{Z} : 
\begin{array}{l}
r>0 \mbox{ or } \\
r=0, c>0 \mbox{ or } \\
r=c=0, m>0
\end{array}  \right\}. 
\end{align*}
By replacing $\Coh(X)$ by 
$\Coh_{\rm{c}}(U)$ in Subsection~\ref{subsec:genDT}, 
we obtain the invariant
\begin{align*}
\DT(r, c, m) \in \mathbb{Q}
\end{align*}
which counts 
Gieseker semistable 
sheaves\footnote{Indeed
such sheaves are scheme theoretically supported
on the zero section of $\pi$, 
if $r>0$ or $r=0$, $c>0$. 
See~\cite[Lemma~2.3]{TodS3}}
$F \in \Coh_{\rm{c}}(U)$
satisfying
\begin{align*}
\ch(\pi_{\ast}F)=(r, c, m). 
\end{align*} 
Here $(r, c, m) \in \Lambda$, and 
$\pi$ is the projection (\ref{pi}).
We also define $\DT(r, c, m)$ for 
elements $(r, c, m) \notin \Lambda$
by
\begin{align*}
\DT(r, c, m) \cneq \left\{ \begin{array}{cc}
\DT(-r, -c, -m) & (r, c, m) \in -\Lambda \\
0 & \pm (r, c, m) \notin \Lambda. 
\end{array} \right. 
\end{align*}
\begin{rmk}
If $(r, c, m) \in -\Lambda$, 
then $\DT(r, c, m)$ counts
objects $F[1] \in D^b(\Coh_0(U))$ for
semistable sheaves $F \in \Coh_0(U)$ with 
$\ch(\pi_{\ast}F)=-(r, c, m)$. 
\end{rmk}
\begin{rmk}\label{rmk:Bog}
By Bogomolov inequality, the invariant 
$\DT(r, c, m)$ is non-zero only if 
$c^2 \ge 2rm$. 
\end{rmk}
\begin{rmk}\label{rmk:coin}
If $r>0$, the invariant $\DT(r, c, m)$ 
coincides with the one after replacing the Gieseker 
stability by slope stability. 
See~\cite[Lemma~2.10]{TodS3}.
\end{rmk}
\begin{rmk}
For a Gieseker semistable sheaf $F \in \Coh_{\rm{c}}(U)$, 
the sheaf $F \otimes \oO_U(\pm 1)$ is also Gieseker semistable. 
This implies
\begin{align}\label{eDT}
\DT(r, c, m)=\DT(r, c+r, m+c+r/2). 
\end{align}
\end{rmk}
We also have the following lemma: 
\begin{lem}\label{lem:DT:dual}
We have the equality
\begin{align}\label{DT:dual}
\DT(r, c, m)=\DT(r, -c, m). 
\end{align}
\end{lem}
\begin{proof}
If $r=0$ and $c>0$, then the result follows since 
$F \mapsto \eE xt_{\mathbb{P}^2}^1(F, \oO_{\mathbb{P}^2})$
gives an isomorphism 
between $\mM^{ss}(0, c, m)$ and $\mM^{ss}(0, c, -m)$. 
Hence we may assume that $r>0$. 
Let us set $\mu=c/r$ and
\footnote{The slope function on $F \in \Coh_{\rm{c}}(U)$
is defined by that of $\pi_{\ast}F$. 
See Subsection~\ref{subsec:tiltcoh}.}
\begin{align*}
\bB_{\mu} \cneq \left\langle 
F, \Coh_0(U)[-1] : 
\begin{array}{l}
F \in \Coh_{\rm{c}}(U)
\mbox{ is slope } \\
\mbox{semistable with slope }
\mu \end{array}
\right\rangle_{\rm{ex}}. 
\end{align*}
Let $H(U)$ be the Hall algebra of $\Coh_{\rm{c}}(U)$. 
Similarly to Subsection~\ref{subsec:genDT}, 
the moduli stacks of objects in $\bB_{\mu}$, 
slope semistable sheaves
in $\Coh_{\rm{c}}(U)$
with slope $\mu$, objects in $\Coh_0(U)[-1]$
determine 
elements 
\begin{align*}
\delta_{\bB_{\mu}},
\delta_{\mM_{\mu}}, 
\delta_0 \in \widehat{H}(U)
\end{align*}
respectively. 
Here $\widehat{H}(U)$ is a suitable completion of $H(U)$. 
By the definition of $\bB_{\mu}$, we have the following relation
in $\widehat{H}(U)$
\begin{align}\label{rel:delta}
(1+\delta_{\bB_{\mu}})=(1+\delta_{\mM_{\mu}}) \ast (1+\delta_{0}). 
\end{align}
Let $\Psi$ be the autoequivalence of 
$D^b \Coh_{\rm{c}}(U)$, given by 
\begin{align*}
\Psi(F)= \dR \hH om_{U}(F, \oO_U(3))[1].
\end{align*}
Then an object $F \in D^b \Coh_{\rm{c}}(U)$
satisfies 
$\ch(\pi_{\ast}F)=(r, c, m)$
if and only if $\ch(\pi_{\ast}\Psi(F))=(r, -c, m)$. 
Moreover the proof of~\cite[Lemma~9.1]{TodK3}
shows that $\Psi$ restricts to the
equivalence between 
$\bB_{-\mu}$ and $\bB_{\mu}$. 
Replacing $\mu$ by $-\mu$ in (\ref{rel:delta}) and applying 
$\Psi_{\ast}$, we obtain  
the following relation
\begin{align}\notag
1+\delta_{\bB_{\mu}} &=\Psi_{\ast}(1+\delta_{\bB_{-\mu}}) \\
\label{rel:Psi}
&=(1+\delta_0) \ast \Psi_{\ast}(1+\delta_{\mM_{-\mu}}). 
\end{align}
We set $\epsilon_{\mM_{\mu}}=\log(1+ \delta_{\mM_{\mu}})$
and $\epsilon_{0}=\log (1+\delta_{0})$. 
The relations (\ref{rel:delta}), (\ref{rel:Psi})
imply that
\begin{align*}
\Psi_{\ast}\epsilon_{\mM_{-\mu}}=\log
\left(\exp(-\epsilon_0) \ast \exp(\epsilon_{\mM_{\mu}})
\ast \exp(\epsilon_0)   \right). 
\end{align*}
By the Baker-Campbell-Hausdorff formula, 
we obtain 
\begin{align}\label{commu}
\Psi_{\ast}\epsilon_{\mM_{-\mu}}
=\epsilon_{\mM_{\mu}}+
\left( \begin{array}{c}
\mbox{ multiple commutators } \\
\mbox{ of }
\epsilon_{\mM_{\mu}} 
\mbox{ and }
\epsilon_{0}
\end{array} \right).
\end{align}
Because $\chi(F, F')=0$ for 
any $F, F' \in \bB_{\mu}$, 
the multiple commutator parts of 
(\ref{commu}) vanish after applying the integration
map (\ref{PiHall}) for $H(U)$. 
By Remark~\ref{rmk:coin}, 
 we obtain the desired identity (\ref{DT:dual}). 
\end{proof}

\subsection{The invariants $\DT(r, c, m)$ for $r>0$}
In this subsection, we review the work of~\cite{TodS3}
\footnote{More recently Manschot~\cite{Man3}
described a closed formula of the Betti numbers of 
moduli spaces of semistable sheaves on $\mathbb{P}^2$.}
on the invariants $\DT(r, c, m)$ with $r>0$. 
For $r>0$, we define the following generating series:  
\begin{align}\label{intro:gen}
\DT(r, c) \cneq \sum_{m}
\DT(r, c, m)(-q^{\frac{1}{2r}})^{c^2 -2rm}. 
\end{align}
If $(r, c)$ is coprime, then
the series (\ref{intro:gen}) is the 
generating series 
of Euler numbers of moduli spaces of stable 
sheaves on $\mathbb{P}^2$
studied in the context of 
Vafa-Witten's S-duality conjecture~\cite{VW}. 
We give some examples of the series (\ref{intro:gen}). 
\begin{exam}
If $r=1$, the work of G\"ottsche~\cite{Got} on Hilbert schemes
of points yields
\begin{align}\label{DT(1,s)}
\DT(1, c)=q^{\frac{\chi(S)}{24}}\eta(q)^{-\chi(S)}. 
\end{align} 
Here $\eta(q)=q^{1/24}\prod_{k\ge 1}(1-q^k)$ is the Dedekind eta function.
\end{exam}
\begin{exam}
The case of $r=2$ and $c=1$
has been studied in several articles~\cite{Kly}, \cite{YosB}, \cite{Yo1}, 
\cite{GoTheta}, \cite{Man2}, 
\cite{BrMan}, \cite{Kool}. 
From these articles, one can show that
\begin{align}\notag
\DT(2, 1)=q^{\frac{1}{4}}\eta(q)^{-6} 
\frac{1}{\sum_{k\in \mathbb{Z}}q^{k^2}} 
\sum_{\begin{subarray}{c}
(a, b) 
\in (0, 1/2)+\mathbb{Z}^2 \\
a>b>0 
\end{subarray}}
(2a-6b)q^{a^2-b^2}. 
\end{align} 
\end{exam}
The result of~\cite[Proposition~3.6]{TodK3} shows that
the series $\DT(r, c)$ satisfies a certain 
recursion formula in terms of modular forms and 
theta type series for indefinite lattices. 
In the notation of~\cite[Proposition~3.6]{TodK3}, 
the recursion formula is
\begin{align}\notag
\DT(r, c)=&\sum_{\begin{subarray}{c}
k\ge 2, \ r_1, \cdots, r_k \in \mathbb{Z}_{\ge 1} \\
r_1+ \cdots + r_k=r
\end{subarray}}
\sum_{\begin{subarray}{c}
\overline{\beta}_i =(c_i, a_i) \in \NS_{<r_i}(\widehat{\mathbb{P}}^2) \\
1\le i\le k
\end{subarray}}
\sum_{G \in G(k)} \frac{(-1)^{k}}{2^{k-1}} \\
\label{recursion}
&\cdot 
U_{(r_1, \overline{\beta}_1), \cdots, (r_k, \overline{\beta}_k)}^{c, G}(q) 
\cdot
\vartheta_{r, 1-c}(q)^{-1}
\prod_{i=1}^{k} \vartheta_{r_i, a_i}(q) \cdot
\prod_{i=1}^{k} \DT(r_i, c_i). 
\end{align}
We briefly explain 
the notation used in (\ref{recursion}).
Let 
\begin{align*}
\rho \colon \widehat{\mathbb{P}}^2 \to \mathbb{P}^2
\end{align*}
be the blow-up at a point
and $C$ the exceptional locus of $\rho$. 
We identity $\mathbb{Z}^2$ with $\NS(\widehat{\mathbb{P}}^2)$ via 
$(x, y) \mapsto x \cdot \rho^{\ast}h + y \cdot [C]$. 
The set $\NS_{<r}(\widehat{\mathbb{P}}^2)$ is a finite 
subset of $\NS(\widehat{\mathbb{P}}^2)$
given by 
\begin{align*}
\NS_{<r}(\widehat{\mathbb{P}}^2)\cneq \{(x, y)\in \mathbb{Z}^2 : 
0\le x<r, \ 0\le y <r\}. 
\end{align*}
The set $G(k)$ is the finite set of certain oriented 
graphs, defined by
\begin{align}\label{def:G(k)}
G(k) \cneq \left\{ \begin{array}{l}
\mbox{ connected and simply connected oriented graphs } \\
\mbox{ with vertex } \{1, 2, \cdots, k\}, i\to j
\mbox{ implies } i<j 
\end{array} \right\}. 
\end{align}
We also have the series (cf.~\cite[Definition~3.2]{TodK3})
\begin{align}\notag
U_{(r_1, \overline{\beta}_1), \cdots, (r_k, \overline{\beta}_k)}^{c, G}(q)
\cneq \sum_{\begin{subarray}{c}
\beta_i \in \NS(\widehat{\mathbb{P}}^2), \
\beta_i \equiv \overline{\beta}_i \ (\text{\rm{mod} } r_i) \\
\beta_1 + \cdots + \beta_k=(c, 1-c)
\end{subarray}}
U(\{(r_i, \beta_i)\}_{i=1}^{m}, H, F_{+}) \\
\label{series:U}
\cdot \prod_{i \to j \text{ \rm{in} } G} K_{\widehat{\mathbb{P}}^2}(r_j \beta_i - r_i \beta_j) q^{-\sum_{1\le i<j \le k} \frac{(r_j \beta_i -r_i \beta_j)^2}{2r r_i r_j}}. 
\end{align}
The coefficient $U(\{(r_i, \beta_i)\}_{i=1}^{k}, H, F_{+}) \in \mathbb{Q}$
is Joyce's combinatorial coefficient 
(cf.~\cite[Definition~4.4]{Joy4}), which is 
complicated but explicit. 
The series (\ref{series:U}) 
is a sum of certain theta type series for indefinite lattices
which converges absolutely for $\lvert q \rvert <1$
(cf.~\cite[Proposition~3.3]{TodS3}). 
Finally the series  
$\vartheta_{r, a}(q)$ is a classical theta series
 \begin{align}\notag
\vartheta_{r, a}(q) \cneq \sum_{\begin{subarray}{c}
(k_1, \cdots, k_{r-1}) \in  (a/r, \cdots, a/r) + \mathbb{Z}^{r-1}
\end{subarray}}
q^{\sum_{1\le i\le j\le r-1}k_i k_j}. 
\end{align} 

The three sums in (\ref{recursion}) are finite sums, and 
$r_i \in \mathbb{Z}_{\ge 1}$
in the RHS of (\ref{recursion}) satisfies $r_i<r$. 
By the induction on $r$ and the formula 
(\ref{DT(1,s)}) for $r=1$, 
we are in principle able to compute $\DT(r, c, m)$ for any $r\ge 1$
and $c, m$. 
By the convergence of the series (\ref{series:U}), 
the series (\ref{intro:gen}) also converges
absolutely for $\lvert q \rvert <1$.  

\subsection{The invariants $\DT(0, c, m)$ with $c>0$}
The $r=0$ case was not treated in~\cite{TodK3}. 
In this case, we can describe $\DT(0, c, m)$ with $c>0$
in terms of stable pair invariants on $X$
with curve classes proportional to $[l]$, i.e. 
stable pair invariants on 
$\omega_{\mathbb{P}^2}$ (cf.~Remark~\ref{rmk:supportD}).   
Using the results of~\cite{Tolim2}, \cite{BrH}, 
we have the following lemma: 
\begin{lem}\label{rankzero}
For $c>0$ and $3c+2m \neq 0$, 
we have the following identity
\begin{align*}
&\DT(0, c, m)= \\
&\sum_{\begin{subarray}{c}
k\ge 1, (c_j, n_j) \in \mathbb{Z}^{2}, 
1\le j\le k \\
(c_1, n_1)+ \cdots +(c_k, n_k) 
=(c, \frac{3}{2}c+m)
\end{subarray}}
\frac{2(-1)^{\frac{3}{2}c+m-k}}{(3c+2m)k}
\left(\prod_{j=1}^{k} P_{n_j, c_j[l]}(X)
-\prod_{j=1}^{k}P_{-n_j, c_j[l]}(X) \right). 
\end{align*}
\end{lem}
\begin{proof}
For $n\in \mathbb{Z}$ and $\beta \in H_2(X, \mathbb{Z})$, 
let $N_{n, \beta}(X) \in \mathbb{Q}$ be the generalized 
DT invariant on $X$ given by\footnote{The invariant
$N_{n, \beta}(X)$ does not depend on 
a choice of an ample divsor $H$. See~\cite[Theorem~6.16]{JS}.}
\begin{align*}
N_{n, \beta}(X) \cneq \DT(0, 0, \beta, n). 
\end{align*}
Since any one dimensional sheaf $F$ on $X$
with $[F]=c[l]$ on $X$
is supported on 
$D$, we have the equality
\begin{align}\label{DT=N}
\DT(0, c, m)=N_{\frac{3}{2}c+m, c[l]}(X). 
\end{align}
By the result of~\cite{Tolim2}, \cite{BrH}, 
we have the following formula
\begin{align}\label{TB}
&1+\sum_{n\in \mathbb{Z}, c>0}
P_{n, c[l]}(X)q^n t^{c}
= \\ \notag
&\prod_{n>0, c>0}\exp \left(
(-1)^{n-1}N_{n, c[l]}(X) q^n t^{c}  \right)^n 
\left( \sum_{n, c}L_{n, c[l]}(X)q^n t^{c} \right). 
\end{align}
Here $L_{n, c[l]}(X) \in \mathbb{Q}$ is a certain 
invariant, which satisfies $L_{n, c[l]}(X)=L_{-n, c[l]}(X)$
and zero for $\lvert n \rvert \gg 0$ for fixed $c$. 
By taking the logarithm of (\ref{TB}), 
replacing $q$ by $q^{-1}$ and taking the difference, we obtain the formula: 
\begin{align*}
&\sum_{n\in \mathbb{Z}, c>0}(-1)^{n-1}nN_{n, c[l]}(X)q^n t^{c}
= \\
&\log \left(1+\sum_{n\in \mathbb{Z}, c>0}
P_{n, c[l]}(X)q^n t^{c}  \right)
-\log \left(1+\sum_{n\in \mathbb{Z}, c>0}
P_{n, c[l]}(X)q^{-n} t^{c}  \right). 
\end{align*}
Here we have used the fact that $N_{n, c[l]}(X)=N_{-n, c[l]}(X)$
(cf.~\cite[Lemma~4.3 (i)]{Tolim2}). 
Combined with (\ref{DT=N}), we obtain the desired identity. 
\end{proof}

\begin{rmk}
By (\ref{eDT}), one can replace $\DT(0, c, m)$
with $\DT(0, c, m+kc)$ for $k\in \mathbb{Z}$
so that $3c+2(m+kc) \neq 0$
holds. 
Then applying Lemma~\ref{rankzero}, one can 
describe $\DT(0, c, m)$ in terms of 
$P_{n, c[l]}(X)$ even if $3c+2m=0$. 
\end{rmk}

\section{The space of weak stability conditions}\label{sec:weak}
Let 
$X$ be a smooth projective Calabi-Yau 3-fold
containing a divisor $D \cong \mathbb{P}^2$. 
In this section, we 
construct a one parameter family 
of weak stability conditions on a triangulated 
category $\dD_{X/Y}$ 
associated to the
birational contraction (\ref{fXY}). 
\subsection{Tilting of $\Coh_{\le d}(X/Y)$}\label{subsec:tiltcoh}
For $0\neq F \in \Coh(\mathbb{P}^2)$, let 
$\mu(F)$ be its slope given by 
\begin{align*}
\mu(F) \cneq \frac{c_1(F) \cdot h}{\rank(F)} \in \mathbb{Q} \cup \{\infty\}. 
\end{align*}
Here $h= c_1(\oO_{\mathbb{P}^2}(1))$
and $\mu(F)=\infty$ if $\rank(F)=0$. 
The above slope function defines the notion of 
$\mu$-semistable sheaves on $\mathbb{P}^2$ in the usual way. 

Let 
$\Coh_{\le d}(X/Y)$
be the abelian subcategory of $\Coh(X)$
defined by (\ref{Coh1}). 
We define the pair of subcategories 
$(\tT_{\le d}, \fF)$ on $\Coh_{\le d}(X/Y)$ in the following way: 
\begin{align}\label{def:TF}
&\tT_{\le d}\cneq \left\langle \Coh_{\le d}(X), i_{\ast}T :
\begin{array}{c}
T\in \Coh(\mathbb{P}^2) \mbox{ is } 
\mu \mbox{-semistable } \\
\mbox{ with } \mu(F)>-1/2
\end{array} \right\rangle_{\rm{ex}} \\
\notag
&\fF\cneq \left\langle i_{\ast}F :
\begin{array}{c}
F\in \Coh(\mathbb{P}^2) \mbox{ is } 
\mu \mbox{-semistable } \\
\mbox{ with } \mu(F)\le -1/2
\end{array} \right\rangle_{\rm{ex}}. 
\end{align}
In what follows, we assume that 
$d \in \{0, 1\}$. 
We have the following lemma: 
\begin{lem}
The subcategories $(\tT_{\le d}, \fF)$ form
a torsion pair of $\Coh_{\le d}(X/Y)$, i.e. 
$\Hom(\tT_{\le d}, \fF)=0$ and 
any object $E \in \Coh_{\le d}(X/Y)$ fits into 
an exact sequence
\begin{align}\label{TEF}
0 \to T \to E \to F \to 0
\end{align}
for $T\in \tT_{\le d}$, $F \in \fF$. 
\end{lem}
\begin{proof}
The condition $\Hom(\tT_{\le d}, \fF)=0$ is obvious from 
the definition of $(\tT_{\le d}, \fF)$. 
We check the condition (\ref{TEF}). 
For $E \in \Coh_{\le d}(X/Y)$, there 
is an exact sequence of sheaves
\begin{align}\label{TEF'}
0 \to T' \to E \to F' \to 0
\end{align}
where $T' \in \Coh_{\le d}(X)$ and 
$F'$ is pure two dimensional (if it is non-zero)
supported on $D$. 
Note that any pure two dimensional semistable sheaf on $X$ 
supported on $D$ is scheme theoretically supported on $D$
(cf.~\cite[Lemma~2.3]{TodK3}). 
Therefore by truncating the Harder-Narasimhan filtration of $F'$, 
and combining the exact sequence (\ref{TEF'}), 
we obtain the desired exact sequence (\ref{TEF}). 
\end{proof}
By taking the tilting (cf.~\cite{HRS}) with respect to the torsion 
pair $(\tT_{\le d}, \fF)$, we obtain the 
heart of a bounded t-structure on $D^b \Coh_{\le d}(X/Y)$
\begin{align*}
\bB_{\le d} \cneq \langle \fF[1], \tT_{\le d} \rangle_{\rm{ex}} 
\subset D^b \Coh_{\le d}(X/Y). 
\end{align*}
Note that we have 
\begin{align*}
\bB_0= \bB_{\le 1} \cap D^b \Coh_0(X/Y). 
\end{align*}
\subsection{Relation of $\bB_{\le d}$ and $\Coh_{\le d}(\yY)$}
Let $\Phi$ be the derived equivalence (\ref{Phi}).
Note that 
$\Phi(\Coh_{\le d}(\yY))$
is the heart of a bounded t-structure on $D^b \Coh_{\le d}(X/Y)$
by the diagram (\ref{cohd}). 
We set
\begin{align}\label{def:Tdag}
&\tT^{\dag} \cneq 
\Phi(\Coh_{\le d}(\yY)) \cap \bB_{\le d}[1] \\
\notag
&\fF^{\dag}_{\le d} \cneq 
\Phi(\Coh_{\le d}(\yY)) \cap \bB_{\le d}. 
\end{align} 
Note that $\tT^{\dag}$ does not depend on a choice
of $d \in \{0, 1\}$. 
\begin{prop}\label{wtor}
The pair $(\tT^{\dag}, \fF^{\dag}_{\le d})$ is 
a torsion pair of $\Phi(\Coh_{\le d}(\yY))$
such that
\begin{align*}
\bB_{\le d}=\langle \fF^{\dag}_{\le d}, 
\tT^{\dag}[-1]  \rangle_{\rm{ex}}. 
\end{align*}
\end{prop}
\begin{proof}
Let $\hH_{\bB_{\le d}}^i(\ast)$ be the $i$-th cohomology 
functor on $D^b \Coh_{\le d}(X/Y)$
with respect to the t-structure with heart
$\bB_{\le d}$. 
The claim 
is equivalent to either one the following conditions
(cf.~\cite[Proposition~2.3.2]{BMT}): 
\begin{align}\label{HB0}
&\hH_{\bB_{\le d}}^i(\Phi(F)) =0, \ i\neq -1, 0, 
\mbox{ for any } F \in \Coh_{\le d}(\yY) \\
\label{HB01}
&\hH^i(\Phi^{-1}(F))=0, \ i\neq 0, 1, \mbox{ for any }
F \in \bB_{\le d}.  
\end{align}
We first check (\ref{HB0}) for $d=0$. 
Since $\Coh_0(\yY)$ is the extension closure of 
$\oO_x$ for $x \neq p$ and $S_j$ with $j=0, 1, 2$, 
we may assume that $F$ is either one of the above objects. 
Obviously (\ref{HB0}) is satisfied for $F=\oO_x$ for $x \neq p$. 
We have 
$\Phi(S_j)=T_j$
by Lemma~\ref{lem:ST}, 
and the definition of $T_j$ in (\ref{def:T}) yields 
\begin{align*}
T_0 \in \tT_{0}, \ T_1 \in \fF[1], \ 
T_2 \in \fF[2]. 
\end{align*}
Therefore (\ref{HB0}) is satisfied for $F=S_j$ with 
$0\le j\le 2$. 

Next we prove (\ref{HB01}) for $d=1$. 
By the definition, the category 
$\bB_{\le 1}$
is the extension closure of objects in $\bB_{0}$ and 
objects in $\Coh_{\le 1}(X)$. 
Since (\ref{HB0}) holds for $d=0$, the condition (\ref{HB01})
also holds for $d=0$. 
Therefore it is enough to show that 
the condition (\ref{HB01}) holds for any 
$F \in \Coh_{\le 1}(X)$. 
Since $\hH^i(\Phi^{-1}(F))$ is 
supported on $p \in \yY$
for $i\neq 0$, 
it is enough to check the vanishing of
the following spaces for any 
$0\le j\le 2$ and $k>0$:
\begin{align*}
&\Hom(S_j[k], \Phi^{-1}(F))
\cong \Ext_X^{-j-k}(T_j[-j], F)\\ 
&\Hom(\Phi^{-1}(F), S_j[-k-1]) \cong 
\Ext^{j-k-1}_X(F, T_j[-j]).  
\end{align*}
Here we have used Lemma~\ref{lem:ST}. 
Since $T_j[-j]$ is a pure two dimensional sheaf, 
the above spaces vanish. 
Therefore we obtain (\ref{HB01}) for $d=1$. 
\end{proof}

We have the following corollary of the 
above proposition: 
\begin{cor}\label{cor:pure}
For any pure one dimensional $F \in \Coh_{\le 1}(\yY)$, 
we have $\Phi(F) \in \bB_{\le 1}$. 
\end{cor}
\begin{proof}
If $F \in \Coh_{\le 1}(\yY)$ is pure, we have
$\Hom(\tT^{\dag}, \Phi(F))=0$. 
By Proposition~\ref{wtor}, we have 
$\Phi(F) \in \fF^{\dag}_{\le 1} \subset \bB_{\le 1}$. 
\end{proof}

\subsection{Abelian category $\aA_{X/Y}$}
We define the triangulated category $\dD_{X/Y}$
in the following way: 
\begin{align*}
\dD_{X/Y} \cneq 
\langle \oO_X, D^b \Coh_{\le 1}(X/Y) \rangle_{\mathrm{tr}}
\subset D^b \Coh(X).
\end{align*} 
The above triangulated category plays a crucial 
role in our main purpose. 
Let $\aA_{X/Y}$ be the subcategory of $\dD_{X/Y}$, 
defined by  
\begin{align*}
\aA_{X/Y} \cneq
\langle \oO_X, \bB_{\le 1}[-1] \rangle_{\rm{ex}} \subset \dD_{X/Y}. 
\end{align*}
We have the following lemma: 
\begin{lem}\label{lem:t}
There is a bounded t-structure on $\dD_{X/Y}$ whose 
heart is given by $\aA_{X/Y}$. In particular, $\aA_{X/Y}$ 
is an abelian category. 
\end{lem}
\begin{proof}
Let $\fF'$ be the subcategory of $\Coh(X)$
defined by
\begin{align*}
\fF' \cneq \{ E \in \Coh(X) : 
\Hom(\tT_{\le 1}, E)=0\}. 
\end{align*}
Since $\tT_{\le 1} \subset \Coh(X)$ is closed under quotients
and $\Coh(X)$ is noetherian,
the pair $(\tT_{\le 1}, \fF')$ forms a torsion pair 
of $\Coh(X)$ (cf.~\cite[Lemma~2.15 (i)]{Tcurve2}). 
We set 
\begin{align}\label{def:C}
\aA \cneq \langle \fF', \tT_{\le 1}[-1] \rangle_{\rm{ex}} \subset D^b \Coh(X). 
\end{align}
Note that $\aA$ is the heart of a bounded t-structure on $D^b \Coh(X)$, 
and 
\begin{align*}
\aA \cap D^b \Coh_{\le 1}(X/Y)=\bB_{\le 1}[-1].
\end{align*}
We apply Proposition~\ref{prop:t} below for $\dD=D^b \Coh(X)$, 
$L=\oO_X$, $\dD'=D^b \Coh_{\le 1}(X/Y)$
and $\aA'=\bB_{\le 1}[-1]$. 
It is obvious that $\aA'$ is closed under subobjects and 
quotients in 
$\aA$. The condition (\ref{vanish:F})
below
is equivalent to the vanishings
\begin{align*}
\Hom(\oO_X, \fF)=0, \ \Hom(\tT_{\le 1}[-1], \oO_X)=0. 
\end{align*}
The vanishing of $\Hom(\oO_X, \fF)$
is equivalent to 
the vanishing of $\Hom(\oO_X, i_{\ast}F)$
for any 
$\mu$-semistable 
$F \in \Coh(\mathbb{P}^2)$
with $\mu(F)\le -1/2$. 
This is equivalent to 
the vanishing of 
$\Hom(\oO_D, F)$, 
which 
follows from
the $\mu$-semsitability 
of $F$ and the inequalities
\begin{align*}
\mu(\oO_D)=0>-\frac{1}{2} \ge \mu(F).
\end{align*} 
The vanishing of $\Hom(\tT_{\le 1}[-1], \oO_X)$
 is equivalent to 
the vanishings 
\begin{align}\notag
\Hom(G[-1], \oO_X)=0, \ 
\Hom(i_{\ast}T[-1], \oO_X)=0
\end{align}
for any $G \in \Coh_{\le 1}(X)$ and 
$\mu$-semistable 
$T \in \Coh(\mathbb{P}^2)$
with $\mu(T)>-1/2$. 
The vanishing of $\Hom(G[-1], \oO_X)$
is equivalent to the vanishing 
of $H^2(X, G)$, which is obvious.
The vanishing of $\Hom(i_{\ast}T[-1], \oO_X)$
is equivalent to  
the vanishing of $\Hom(T, \omega_{\mathbb{P}^2})$, 
which 
follows from the 
$\mu$-semistability of $T$
and the inequalities 
\begin{align*}
\mu(T)>-\frac{1}{2} >-3=\mu(\omega_{\mathbb{P}^2}).
\end{align*}
\end{proof}
We have used the following result proved in~\cite{Tcurve1}: 
\begin{prop}\emph{(\cite[Proposition~3.6]{Tcurve1})}\label{prop:t}
Let $\dD$ be a $\mathbb{C}$-linear triangulated category 
and $\aA \subset \dD$ the heart of a bounded t-structure on $\dD$. 
Take $L \in \aA$ such that $\End(L)=\mathbb{C}$ and 
a full triangulated subcategory $\dD' \subset \dD$ satisfying the 
following conditions:
\begin{itemize}
\item The category $\aA' \cneq \aA \cap \dD'$ is the heart of a bounded
t-structure on $\dD'$, which is closed under subobjects and quotients 
in $\aA$. 
\item For any object $F \in \aA'$, we have
\begin{align}\label{vanish:F}
\Hom(L, F)=\Hom(F, L)=0.
\end{align}
\end{itemize}
Let $\dD_{L} \subset \dD$ be the triangulated subcategory 
defined by 
$\dD_{L} \cneq \langle L, \dD' \rangle_{\rm{tr}}$. 
Then $\aA_{L} \cneq \dD_{L} \cap \aA$ is the heart of a bounded t-structure on $\dD_{L}$, which satisfies 
$\aA_{L}=\langle L, \aA' \rangle_{\rm{ex}}$. 
\end{prop}

Let $\cC \subset D^b \Coh(X)$ be the subcategory 
defined by 
\begin{align*}
\cC \cneq \langle \oO_X(rD), \Coh_{\le 1}(X)[-1]: 
r\in \mathbb{Z}
 \rangle_{\rm{ex}}. 
\end{align*}
We will use the following 
lemmas on the 
abelian category $\aA_{X/Y}$. 
\begin{lem}\label{lem:cC}
The category $\cC$ is a subcategory of $\aA_{X/Y}$. 
In particular, for any $r \in \mathbb{Z}$
and a stable pair $(\oO_X \stackrel{s}{\to}F)$ on $X$, we have
\begin{align}\label{O(rD)}
\oO_X(rD) \otimes (\oO_X \stackrel{s}{\to} F) \in \aA_{X/Y}. 
\end{align}
\end{lem}
\begin{proof}
It is enough to check that 
$\oO_X(rD) \in \aA_{X/Y}$
for any $r\in \mathbb{Z}$ to show 
the inclusion $\cC \subset \aA_{X/Y}$. 
Obviously this holds for $r=0$. If $r>0$, we have 
the distinguished triangle
\begin{align*}
\oO_X((r-1)D) \to \oO_X(rD) \to \oO_D(-3r). 
\end{align*}
Since $\oO_D(-3r) \in \fF \subset \aA_{X/Y}$, 
we have $\oO_X(rD) \in \aA_{X/Y}$ by the induction of $r$. 
If $r<0$, 
we have 
the distinguished triangle
\begin{align*}
\oO_D(-3r-3)[-1] \to
\oO_X(rD) \to \oO_X((r+1)D). 
\end{align*}
Since $\oO_D(-3r-3)[-1] \in \tT_{\le 1}[-1] \subset \aA_{X/Y}$, 
we have $\oO_X(rD) \in \aA_{X/Y}$ by the induction of $r$. 
Hence $\cC \subset \aA_{X/Y}$ holds. 
By the distinguished triangle
\begin{align*}
\oO_X(rD) \otimes F[-1]
 \to \oO_X(rD) \otimes (\oO_X \stackrel{s}{\to} F)
\to \oO_X(rD)
\end{align*}
the object (\ref{O(rD)}) is an object in $\cC$, 
hence an object in $\aA_{X/Y}$.  
\end{proof}

\begin{lem}\label{lem:filt}
For any $E \in \aA_{X/Y}$ with $\rank(E)=1$, there is 
a filtration $E_1 \subset E_2 \subset E_3=E$
in $\aA_{X/Y}$ such that
\begin{align}\label{filt1}
E_1 \in \fF, \ E_2/E_1 \in \cC, \ 
E/E_2 \in \tT_{\le 1}^{\rm{pure}}[-1]. 
\end{align}
Here $\tT_{\le 1}^{\rm{pure}}$
is the subcategory of $\tT_{\le 1}$ consisting of 
pure two dimensional sheaves. 
\end{lem}
\begin{proof} 
Recall that the category $\aA_{X/Y}$
is a subcategory of $\aA$ defined by (\ref{def:C}). 
By the definition of $\aA$, there is an 
exact sequence in $\aA$
\begin{align}\label{FET-1}
0 \to F \to E \to T[-1] \to 0
\end{align}
for $F \in \fF'$ and $T \in \tT_{\le 1}$.
We have the exact sequences of sheaves
\begin{align*}
&0 \to F' \to F \to F'' \to 0 \\
&0 \to T' \to T \to T'' \to 0
\end{align*}
where $F'$ is the torsion part of $F$
and $T'$ is the maximal 
subsheaf of $T$ contained in $\Coh_{\le 1}(X)$. 
Since $E \in \aA_{X/Y}$,
the sheaf $F'$ is supported on $D$, 
and satisfies $\Hom(\tT_{\le 1}, F')=0$.
It follows that $F' \in \fF$, 
and also 
$T'' \in \tT_{\le 1}^{\rm{pure}}$
holds by the construction of $T'$. 
The sheaf $F''$ is 
a torsion free rank one sheaf on $X$, 
whose determinant is trivial outside $D$. 
Hence $F''$ is written as $\oO_X(rD) \otimes I_C$
for some $r\in \mathbb{Z}$ and a subscheme $C \subset X$
with $\dim C \le 1$. 
We set $E_1=F'$ and 
$E_2$ to be the kernel
of the composition in $\aA$
\begin{align*}
E \twoheadrightarrow T[-1] \twoheadrightarrow T''[-1].
\end{align*}
Then $E_2/E_1$ fits into the exact sequence in $\aA$
\begin{align*}
0 \to F'' \to E_2/E_1 \to T'[-1] \to 0. 
\end{align*}
Since $F'', T'[-1] \in \cC$, we have
$E_2/E_1 \in \cC$, showing that $E_{\bullet}$
is a desired filtration. 
\end{proof}
We also have the following positivity lemma: 
\begin{lem}\label{lem:positive}
For any $E \in \aA_{X/Y}$, we have 
$\rank(E)\ge 0$
and 
\begin{align}\label{positive}
-2\ch_2(E) -\frac{2}{3}D \ch_1(E) \ge 0. 
\end{align}
Here for $\beta \in H^4(X)$, 
$\beta>0$ means that 
it is 
a numerical class of an effective integral 
one cycle on
$X$. 
\end{lem}
\begin{proof}
The first statement is obvious. 
For the second statement, 
note that $F \in \Coh(\mathbb{P}^2)$
with $\ch(F)=(r, c, m)$ 
satisfies 
\begin{align*}
\ch_2(i_{\ast}F)+\frac{D}{3} \ch_1(i_{\ast}F)=\left(\frac{r}{2}+c\right)[l]. 
\end{align*}
Therefore the positivity of (\ref{positive}) follows 
from the construction of $\aA_{X/Y}$. 
\end{proof}

\subsection{Abelian category $\aA_{\yY}$}
Let $\yY$ be the orbifold (\ref{yYY}) 
which is derived equivalent to $X$. 
We define 
\begin{align*}
\dD_{\yY} \cneq \langle \oO_{\yY}, D^b \Coh_{\le 1}(\yY) 
\rangle_{\rm{tr}} \subset D^b \Coh(\yY). 
\end{align*}
By Lemma~\ref{lem:ST} and the diagram (\ref{cohd}), 
the equivalence (\ref{Phi}) restricts to 
the equivalence
\begin{align*}
\Phi \colon \dD_{\yY} \stackrel{\sim}{\to} \dD_{X/Y}. 
\end{align*}
We define the following subcategory of $\dD_{\yY}$
\begin{align*}
\aA_{\yY} \cneq \langle 
\oO_{\yY}, \Coh_{\le 1}(\yY)[-1] \rangle_{\rm{ex}} \subset \dD_{\yY}. 
\end{align*}
\begin{lem}\label{lem:AY}
There is a bounded t-structure on $\dD_{\yY}$ whose 
heart is given by $\aA_{\yY}$. 
\end{lem}
\begin{proof}
In~\cite[Lemma~3.5]{Tcurve1}, the same
statement was proved 
 for non-orbifold Calabi-Yau 3-folds
using Proposition~\ref{prop:t}. 
The same argument of~\cite[Lemma~3.5]{Tcurve1}
is applied without any modification. 
\end{proof}
\begin{lem}
For an orbifold stable pair 
$(\oO_{\yY} \stackrel{s}{\to} F)$ on $\yY$, we have
\begin{align}\label{AcapA}
\Phi(\oO_{\yY} \stackrel{s}{\to} F) \in \fF^{\sharp} \cneq
 \Phi(\aA_{\yY}) \cap \aA_{X/Y}. 
\end{align}
\end{lem}
\begin{proof}
Since $F$ is pure one dimensional, 
we have $\Phi(F)[-1] \in \fF^{\sharp}$
by  
Corollary~\ref{cor:pure}. 
Then the result follows from 
the distinguished triangle
\begin{align*}
\Phi(F)[-1] \to \Phi(\oO_{\yY} \stackrel{s}{\to} F) \to \oO_X. 
\end{align*}
\end{proof}
The category $\tT^{\dag}$ defined by (\ref{def:Tdag}) 
is a subcategory of 
$\Phi(\Coh_{\le 1}(\yY))$, hence we have
$\tT^{\dag}[-1] \subset \Phi(\aA_{\yY})$. 
The relationship between $\aA_{X/Y}$ and $\aA_{\yY}$ is 
given as follows: 
\begin{lem}\label{lem:AYtor}
The subcategory
$\tT^{\dag}[-1] \subset \Phi(\aA_{\yY})$
fits into a torsion pair $(\tT^{\dag}[-1], \fF^{\sharp})$ 
on $\Phi(\aA_{\yY})$ such that 
\begin{align*}
\aA_{X/Y}=\langle \fF^{\sharp}, \tT^{\dag}[-2] \rangle_{\rm{ex}}. 
\end{align*}
\end{lem}
\begin{proof}
By Proposition~\ref{wtor}, we have
\begin{align*}
\Phi(\aA_{\yY}) \subset
\langle \aA_{X/Y}[1], \aA_{X/Y} \rangle_{\rm{ex}}. 
\end{align*}
Hence $\Phi(\aA_{\yY})$ and $\aA_{X/Y}$
are related by a tilting
for some torsion pair of $\Phi(\aA_{\yY})$. 
The free part is 
$\Phi(\aA_{\yY}) \cap \aA_{X/Y}$, 
which coincides with $\fF^{\sharp}$ by 
its definition (\ref{AcapA}). 
The torsion part is 
\begin{align*}
\Phi(\aA_{\yY}) \cap \aA_{X/Y}[1]&=
\Phi(\Coh_{\le 1}(\yY))[-1] \cap \bB_{\le 1} \\
&=\tT^{\dag}[-1]. 
\end{align*}
Hence we obtain the result.  
\end{proof}

\subsection{Weak stability conditions on $\dD_{X/Y}$}
We construct a one parameter family of weak stability 
conditions on $\dD_{X/Y}$ in the sense of~\cite{Tcurve1}, 
using the t-structures in the previous subsections. 
Let $\Gamma$ be the free abelian group
defined by
\begin{align*}
\Gamma^{} \cneq \Imm (K(\dD_{X/Y}) \to N(X)). 
\end{align*}
The map $\Gamma^{} \to \mathbb{Z}$
sending $F$ to $\rank(F)$ 
has a splitting by $1 \mapsto [\oO_X]$. 
Hence we have 
\begin{align}\label{isom:Gamma}
\Gamma^{} \cong \mathbb{Z} \oplus N_{\le 1}(X/Y) 
\end{align}
and the natural map 
$\cl \colon K(\dD_{X/Y}) \to \Gamma$
given by 
$F \mapsto [F]$ is identified with 
the map 
\begin{align}\label{def:cl}
\cl(F)=(\rank(F), [F]-\rank(F)[\oO_X]). 
\end{align}
We set $\Gamma_0 = N_0(X/Y)$, 
$\Gamma_1 = N_{\le 1}(X/Y)$, so 
that we have the filtration
\begin{align}\label{fGamma}
0=\Gamma_{-1} \subset 
\Gamma_0 \subset \Gamma_1 \subset \Gamma_2=\Gamma^{}. 
\end{align}
The subquotients of the above
filtration are described by the 
isomorphisms (cf.~Remark~\ref{rmk:K}): 
\begin{align}\label{isom:quot}
\rank \colon \Gamma_2/\Gamma_1 \stackrel{\sim}{\to}
\mathbb{Z}, \quad 
f_{\ast} \colon \Gamma_1/\Gamma_0 \stackrel{\sim}{\to} N_{1}(Y), \ 
i_{\ast} \colon K(\mathbb{P}^2) \stackrel{\sim}{\to} \Gamma_0. 
\end{align}
Here $N_{1}(Y) \subset H_2(Y, \mathbb{Z})$
is the subgroup generated by algebraic one cycles on $Y$.

Let us take an element
\begin{align}\label{Zj}
Z=\{Z_j\}_{j=0}^{2} \in 
\prod_{j=0}^{2} 
\Hom_{\mathbb{Z}}(\Gamma_{j}/\Gamma_{j-1}, \mathbb{C}). 
\end{align}
For $v \in \Gamma^{}$,
we take unique $j$ such that $v \in \Gamma_{j} \setminus \Gamma_{j-1}$
and set
\begin{align*}
Z(v) \cneq Z_{j}(\overline{v}) \in \mathbb{C}. 
\end{align*} 
Here $\overline{v} \in \Gamma_{j}/\Gamma_{j-1}$ is the 
class of $v$ in $\Gamma_{j}/\Gamma_{j-1}$. 
For $E \in \dD_{X/Y}$, we write 
$Z(\cl(E))$ just as $Z(E)$ for simplicity. 
\begin{defi}\emph{(\cite{Tcurve1})}
A \textit{weak stability condition} on $\dD_{X/Y}$ with 
respect to the filtration $\Gamma_{\bullet}$ is 
data of $(Z, \aA)$, 
where $Z$ is an element (\ref{Zj}), 
$\aA \subset \dD_{X/Y}$ is the heart 
of a bounded t-structure satisfying 
\begin{align}\label{positive2}
Z(\aA \setminus \{0\}) \subset \mathbb{H} 
\cneq \{ r \exp(\pi \phi \sqrt{-1}) : 
r>0, 0<\phi \le 1\}. 
\end{align}
The data $(Z, \aA)$ should also satisfy 
other technical conditions, called \textit{Harder-Narasimhan 
property} and \textit{support property}
\emph{(cf.~\cite[Section~2]{Tcurve1})}.  
\end{defi}
An object $E \in \aA$ is called 
$Z$-\textit{(semi)stable} if for any 
subobject $0\neq F \subsetneq E$ in $\aA$, we have
the inequality
\begin{align*}
\arg Z(F)<(\le) \arg Z(E/F)
\end{align*}
in $(0, \pi]$.
Similarly to the space of Bridgeland 
stability conditions~\cite{Brs1}, 
the set of weak stability conditions
$\Stab_{\Gamma_{\bullet}}(\dD_{X/Y})$
has a structure of a complex manifold such that
the forgetting map
\begin{align*}
\Stab_{\Gamma_{\bullet}}(\dD_{X/Y})
\to \prod_{j=0}^{2} 
\Hom_{\mathbb{Z}}(\Gamma_{j}/\Gamma_{j-1}, \mathbb{C})
\end{align*}
sending $(Z, \aA)$ to $Z$ 
is a local homeomorphism. 

We fix $\psi \in (\pi/2, \pi)$ and 
an ample divisor $\omega$ on $Y$. 
For $t\in\mathbb{R}$, we construct
\begin{align*}
Z_t=\{Z_{j, t}\}_{j=0}^2
 \in \prod_{j=0}^{2} \Hom_{\mathbb{Z}}(\Gamma_{j}/\Gamma_{j-1}, 
\mathbb{C})
\end{align*}
via the isomorphisms (\ref{isom:quot})
in the following way: 
\begin{align}\notag
&Z_{2, t} \colon \mathbb{Z} \to \mathbb{C}, \
R \mapsto R \cdot \exp(\psi \sqrt{-1}) \\
\notag
&Z_{1, t} \colon N_1(Y) \to \mathbb{C}, \
\beta \mapsto -(\beta \cdot \omega) \sqrt{-1} \\
\label{Z0t}
&Z_{0, t} \colon K(\mathbb{P}^2) \to \mathbb{C}, \ 
F \mapsto \int_{\mathbb{P}^2}e^{h/2-t\sqrt{-1}} \ch(F). 
\end{align}
Here
$h=c_1(\oO_{\mathbb{P}^2}(1))$. 
For $(r, c, m) \in \mathbb{Q}^{3}$, 
we set
\begin{align*}
(\widehat{r}, \widehat{c}, \widehat{m}) \cneq
 \left(r, c+\frac{r}{2}, m+\frac{c}{2}+\frac{r}{8}  \right). 
\end{align*}
If we write $\ch(F)=(r, c, m)$
for $F \in K(\mathbb{P}^2)$
via (\ref{isom:P}), 
then we have
$e^{h/2}\ch(F)=(\widehat{r}, \widehat{c}, \widehat{m})$, 
and $Z_{0, t}(F)$ is written as
\begin{align*}
Z_{0, t}(F)=\widehat{m}-\frac{t^2}{2}\widehat{r}
-t\widehat{c}\sqrt{-1}. 
\end{align*}

\begin{lem}\label{lem:sigmat}
The data
\begin{align*}
\sigma_{t}^{} \cneq (Z_t, \aA_{X/Y}), \quad t>0
\end{align*}
determine a one parameter family of weak stability conditions 
on $\dD_{X/Y}$
with respect to $\Gamma_{\bullet}$. 
\end{lem}
\begin{proof}
We check that (\ref{positive2}) holds. 
For non-zero $E \in \aA_{X/Y}$, 
suppose that $\rank(E) \neq 0$. 
Then $\rank(E)>0$, 
$[E] \in \Gamma_2 \setminus \Gamma_1$
and  
\begin{align*}
Z_t(E)=\rank(E) \cdot \exp(\psi \sqrt{-1}) \in \mathbb{H}.
\end{align*}
If $\rank(E)=0$, then we have 
$E \in \bB_{\le 1}[-1]$. 
If furthermore $E \notin \bB_0[-1]$, then
$[E] \in \Gamma_1 \setminus \Gamma_0$ and 
$-f_{\ast}\ch_2(E)$ is a numerical class of an 
effective one cycle on $Y$. 
Therefore
\begin{align*}
Z_t(E)=-f_{\ast} \ch_2(E) \cdot \omega \sqrt{-1} \in \mathbb{H}. 
\end{align*}
Finally if $0\neq E \in \bB_0[-1]$, then we have 
$Z_t(E)=Z_{0, t}(E)$. 
The construction of Bridgeland stability conditions on
local $\mathbb{P}^2$
in~\cite[Section~4]{BaMa}
shows that 
$(Z_{0, t}, \bB_0[-1])$
determines a Bridgeland stability 
condition on $D^b \Coh_0(X/Y)$. 
Hence we have $Z_{0, t}(E) \in \mathbb{H}$. 
Other properties (Harder-Narasimhan property, 
support property) are checked in a straightforward way. 
For example, the same argument of~\cite[Lemma~3.4]{TodK3} works. 
\end{proof}
We now investigate
the limiting point
$\lim_{t\to +0}\sigma_t$
of the above weak stability conditions. 
The following lemma shows that 
the one parameter family 
in Lemma~\ref{lem:sigmat}
connects the `large volume limit point'
with the `orbifold point', a similar picture 
obtained by Bayer-Macri~\cite{BaMa}
for the space of Bridgeland stability conditions on 
local $\mathbb{P}^2$
 (cf.~Figure~\ref{fig:one} in Subsection~\ref{subsec:idea}). 

\begin{lem}\label{lem:limit}
We have 
\begin{align*}
\sigma_0 \cneq (Z_{0}, \Phi(\aA_{\yY}))\in \Stab_{\Gamma_{\bullet}}(\dD_{X/Y})
\end{align*} 
and it coincides with $\lim_{t\to +0}\sigma_t$. 
\end{lem}
\begin{proof}
Let $T_j$
for $0\le j\le 2$ be the objects given by (\ref{def:T}). 
A direct calculation shows that
\begin{align}\label{comp:T}
&Z_{0, t}(T_0[-1])=-\frac{1}{8}+\frac{t^2}{2}+\frac{t}{2} \sqrt{-1} \\
\notag
&Z_{0, t}(T_1[-1])=-\frac{3}{4}-t^2 \\
\notag
&Z_{0, t}(T_2[-1])=-\frac{1}{8}+\frac{t^2}{2}-\frac{t}{2} \sqrt{-1}.
\end{align}
In particular $\arg Z_{0, 0}(T_j[-1]) =\pi$. 
By Lemma~\ref{lem:ST}, 
this implies that the pair 
$(Z_{0, 0}, \Phi(\Coh_0(\yY))[-1])$
is a Bridgeland stability condition on 
$D^b \Coh_0(X/Y)$. 
This fact together with
 the same argument of Lemma~\ref{lem:sigmat}
show that $\sigma_0$ is an element of 
$\Stab_{\Gamma_{\bullet}}(\dD_{X/Y})$. 
By Lemma~\ref{lem:AYtor}, $\Phi(\aA_{\yY})$ and $\aA_{X/Y}$
differ by a tilting, hence
$\lim_{t\to +0}\sigma_t=\sigma_0$
follows from~\cite[Lemma~7.1]{Tcurve1}. 
\end{proof}

\section{Comparison of stable pair invariants}\label{sec:Comp}
In this section, we relate rank one 
$\sigma_t$-semistable objects for $t\gg 1$ 
with stable pairs on $X$,  
and those for $0<t\ll 1$ with  
orbifold stable pairs on $\yY$. 
We then apply Joyce-Song wall-crossing formula 
to derive a relationship between 
stable pair invariants on $X$
and those on $\yY$. 

\subsection{Moduli stacks of semistable objects}
Let 
$\mM$ be the moduli stack of 
objects $E \in D^b \Coh(X)$ 
satisfying the condition
\begin{align*}
\Ext^{<0}(E, E)=0.
\end{align*}
By the result of Liebich~\cite{LIE}, the 
stack $\mM$ is an Artin stack 
locally of finite type over $\mathbb{C}$. 
For $R\in \mathbb{Z}_{\ge 0}$, 
let  
\begin{align*}
\oO bj^{\le R}(\aA_{X/Y}) \subset \mM
\end{align*}
be the substack of objects $E \in \aA_{X/Y}$
with $\rank(E) \le R$. 
\begin{lem}
The stack $\oO bj^{\le 1}(\aA_{X/Y})$ is an open substack 
of $\mM$. In particular, it is an Artin stack 
locally of finite type over $\mathbb{C}$. 
\end{lem}
\begin{proof}
The proof is similar to~\cite[Lemma~3.15]{Tcurve1}, 
so we just give a brief explanation. 
Let $\aA$ be the 
abelian category defined by (\ref{def:C}). 
By the argument of~\cite[Appendix~A]{AB}, 
the torsion pair on $\Coh(X)$ which 
defines $\aA$
forms a stack of torsion theories, 
which implies that the stack $\oO bj(\aA)$ 
of objects in $\aA$ is an open substack of $\mM$. 
Therefore it is enough to show that the embedding
\begin{align*}
\oO bj^{\le 1}(\aA_{X/Y})
\subset \oO bj(\aA)
\end{align*}
is an open immersion. 
For $E \in \aA_{X/Y}$
with $\rank(E) \le 1$, we have the exact sequence in $\aA$
\begin{align*}
\hH^0(E) \to E \to \hH^1(E)[-1]
\end{align*}
satisfying the following two conditions: 
\begin{itemize}
\item The sheaf $\hH^0(E)$ is torsion free on $X \setminus D$. 
\item The determinant line bundle $\det(E)$ is of the form 
$\oO_X(rD)$
for some $r \in \mathbb{Z}$. 
\end{itemize}
Conversely if an object $E \in \aA$ 
with $\rank(E) \le 1$
satisfies the 
above two conditions, then we have 
$E \in \aA_{X/Y}$. 
The openness of the former condition follows 
from the same spectral sequence 
argument of~\cite[Lemma~3.15, Step1]{Tcurve1}, 
and the latter condition is obviously open. 
\end{proof}

For 
$t\in \mathbb{R}_{>0}$, $R \in \{0, 1\}$
and $\alpha \in N_{\le 1}(X/Y)$, let
\begin{align}\label{stack:M}
\mM_{t}(R, \alpha) \subset \oO bj^{\le 1}(\aA_{X/Y})
\end{align}
be the substack of 
$Z_{t}$-semistable objects $E \in \aA_{X/Y}$
satisfying 
$\cl(E)=(R, \alpha)$, where $\cl$ is the map (\ref{def:cl}). 
\begin{prop}\label{prop:fund}
Suppose that
\begin{align*}
(R, \alpha) \in (0, N_0(X/Y))
\ \mbox{ or } \ (1, N_{\le 1}(X/Y)).
\end{align*}

(i)
 The stack $\mM_{t}^{}(R, \alpha)$
is an Artin stack of finite type over $\mathbb{C}$, 
such that (\ref{stack:M}) is an open immersion. 

(ii) There 
is a finite number of real numbers
\begin{align*}
0=t_0 < t_1< \cdots < t_{k-1}<t_k=\infty
\end{align*}
such that
for $t \in (t_{i-1}, t_i)$, 
the stack $\mM_{t}(R, \alpha)$ is 
constant 
and consists of 
$Z_t$-stable objects. 
\end{prop}
\begin{proof}
The proof follows from the same arguments 
in the author's previous papers~\cite{TodK3}, \cite{Tolim}. 
We don't repeat their details here, 
and just give a brief explanation in the $R=1$ case. 
In both of (i) and (ii), 
we use Lemma~\ref{lem:positive} instead of~\cite[Lemma~2.10]{TodK3}. 

(i) 
Following the proof of~\cite[Lemma~4.13~(ii)]{TodK3},
we can show that the set of objects 
in $\mM_t(1, \alpha)$ is 
bounded. Indeed for
any object $[E] \in \mM_t(1, \alpha)$, 
one can take a
filtration
\begin{align*}
E_1 \subset E_2 \subset E_3=E
\end{align*}
 as in Lemma~\ref{lem:filt}. 
The object $E_2/E_1 \in \cC$ also 
admits a filtration 
\begin{align*}
F_1 \subset F_2 \subset F_3=E_2/E_1
\end{align*}
such that $F_1$ and $F_3/F_2$
are objects in $\Coh_{\le 1}(X)[-1]$, 
and $F_2/F_1=\oO_X(rD)$. 
Using Lemma~\ref{lem:positive}, 
the Bogomolov inequality on $\mathbb{P}^2$
and the  
$Z_t$-semistability of $E$, we can bound the 
numbers of Harder-Narasimhan factors and 
numerical classes of $E_1, E_3/E_2$, $F_1$, $F_3/F_2$. 
This implies the boundedness of objects
in $\mM_t(1, \alpha)$. 
The openness of (\ref{stack:M})
follows from the boundedness of 
semistable objects by the 
same proof of~\cite[Theorem~3.20]{Tolim}.

(ii) We apply the same proof of~\cite[Proposition~9.7]{TodK3}. 
Let $A \in \bB_{\le 1}[-1]$ be an object
satisfying the following conditions: 
\begin{itemize}
\item 
There are $t>0$, $[E] \in \mM_t(1, \alpha)$ and
an injection $A \hookrightarrow E$
or a surjection $E \twoheadrightarrow A$ in $\aA_{X/Y}$. 
\item We have 
$Z_t(A) \in \mathbb{R}_{>0} \exp(\sqrt{-1}\psi)$. 
\end{itemize}
Note that the second condition implies that $A \in \bB_0[-1]$. 
By taking the filtration of $A$ as in Lemma~\ref{lem:filt} 
and using the Bogomolov inequality on $\mathbb{P}^2$, one 
 can show that the possible numerical classes
for such $A$ is a finite
set. In particular, the possible 
$t\in \mathbb{R}$
is also a finite set, giving the desired result. 
\end{proof}
The result of Proposition~\ref{prop:fund} (ii) in particular 
shows that, 
for $t\in \mathbb{R}_{>0} \setminus \{t_1, \cdots, t_{k-1}\}$, 
we have the $\mathbb{C}^{\ast}$-gerby structure
\begin{align}\label{gerby}
\mM_t(1, \alpha) \to M_t(1, \alpha)
\end{align}
for an algebraic space $M_t(R, \alpha)$
of finite type. 
\subsection{DT type invariants}
We define the DT type 
invariants counting 
$Z_t$-semistable objects
$E \in \aA_{X/Y}$
with $\rank(E) \le 1$. 
We first define the rank one invariants. 
Let us take 
\begin{align*}
\alpha \in N_{\le 1}(X/Y), \quad
t\in \mathbb{R}_{>0} \setminus \{t_1, \cdots, t_{k-1}\}
\end{align*}
where $t_i$ is given in 
Proposition~\ref{prop:fund}
for $\mM_t(1, \alpha)$.  
We define
\begin{align*}
\DT_t(1, \alpha) \cneq \int_{M_t(1, \alpha)} \nu \ d\chi. 
\end{align*}
Here $M_t(1, \alpha)$ is the coarse moduli 
space of $\mM_t(1, \alpha)$
given in (\ref{gerby}), 
and $\nu$ is the Behrend function on $M_t(1, \alpha)$. 

In the rank zero case, there may be strictly $\sigma_t$-semistable 
objects even for a general $t$.
So we need to use the Hall algebra
as in Subsection~\ref{subsec:genDT}
to define the invariants. 
Let $H^{0}(\aA_X)$ be the 
stack theoretic Hall algebra 
of rank zero objects in $\aA_X$. 
As a $\mathbb{Q}$-vector space, it is spanned by 
isomorphism classes of symbols
\begin{align*}
[\rho \colon \xX \to \oO bj^0(\aA_X)]
\end{align*}
where $\xX$ is an Artin stack of finite type 
with affine geometric stabilizers
and $\rho$ is a 1-morphism. 
The relation is generated by (\ref{Hall:rel})
after replacing $\cC oh(X)$ by $\oO bj^0(\aA_X)$. 
We also have the associative $\ast$-product 
on $H^0(\aA_X)$, similarly to the $\ast$-product on $H(X)$. 

For $\alpha \in N_{0}(X/Y)$, the stack 
$\mM_t(0, \alpha)$ determines the element
\begin{align*}
\delta_{t}(0, \alpha) \cneq [\mM_t(0, \alpha) \subset \oO bj^{0} (\aA_X)]
 \in H^{0}(\aA_X). 
\end{align*}
We define the element 
$\epsilon_t(0, \alpha) \in H^{0}(\aA_X)$ to be
\begin{align*}
\epsilon_t(0, \alpha) \cneq \sum_{\begin{subarray}{c}
k\ge 1, \alpha_1, \cdots, \alpha_k \in 
N_{0}(X/Y) \\
\arg Z_{0, t}(\alpha_i)= \arg Z_{0, t}(\alpha)
\end{subarray}} \frac{(-1)^{k-1}}{k}
\delta_{t}(0, \alpha_1) \ast \cdots \ast \delta_t(0, \alpha_k). 
\end{align*}
Similarly to (\ref{PiHall}), we have the 
integration map $\Pi$
from the Lie algebra
of elements supported on 
virtual indecomposable objects in $H^0(\aA_X)$
to the Lie algebra
\begin{align*}
C^0(\aA_X)=\bigoplus_{\alpha \in N_0(X/Y)}
\mathbb{Q}\cdot c_{\alpha}
\end{align*}
with bracket given by (\ref{bracket}). 
The element $\epsilon_t(0, \alpha)$ is supported on virtual 
indecomposable objects, and 
the invariant $\DT_t(0, \alpha) \in \mathbb{Q}$
is defined by 
\begin{align*}
\Pi(\epsilon_t(0, \alpha))=-\DT_t(0, \alpha) \cdot c_{\alpha}. 
\end{align*}
It counts $Z_t$-semistable objects
$E \in \aA_X$ with $\cl(E)=(0, \alpha)$.

\subsection{The invariants $\DT_t(R, \alpha)$ for $t\gg 1$ and $0<t \ll 1$}
Let us take $\alpha \in N_{\le 1}(X/Y)$
and its Chern character $\ch(\alpha) \in \mathbb{Q}[D] \oplus 
H^{\ge 4}(X)$. 
Note that we can write $(1, \ch(\alpha)) \in 
H^{\ast}(X, \mathbb{Q})$  as 
\begin{align*}
(1, \ch \alpha)=e^{rD}(1, 0, -\beta, -n)
\end{align*}
for some $r \in \mathbb{Z}$, $\beta \in H^4(X)$ and $n\in H^6(X)$. 
By identifying $H^4(X)$ with $H_2(X)$ and $H^6(X)$ with $\mathbb{Q}$, 
we have the following proposition: 
\begin{prop}\label{prop:DTP}
For $t \gg 0$, we have 
$M_t(1, \alpha) \cong P_n(X, \beta)$. In particular, we have 
$\DT_t(1, \alpha)=P_{n, \beta}(X)$
for $t\gg 0$. 
\end{prop}
\begin{proof}
We first construct a morphism
\begin{align}\label{MP}
M_t(1, \alpha) \to P_n(X, \beta)
\end{align}
for $t \gg 0$. 
Let us take an object $[E] \in M_t(1, \alpha)$
and a filtration in $\aA_{X/Y}$
\begin{align*}
E_1 \subset E_2 \subset E_3 =E
\end{align*}
as in Lemma~\ref{lem:filt}.
If $E_1 \neq 0$, 
it contradicts to the $Z_t$-stability of $E$
for $t\gg 0$ as 
\begin{align*}
\lim_{t \to \infty} \arg Z_{t}(E_1)=\pi
>\psi=\arg Z_t(E). 
\end{align*}
Therefore $E_1=0$, and the same argument 
also shows $E_3/E_2=0$. 
It follows that $E \in \cC$. 
We take the distinguished triangle
\begin{align*}
\hH^0(E) \to E \to \hH^1(E)[-1]
\end{align*}
which is an exact sequence in $\cC$. 
Note that $\hH^0(E)$ is of the form 
$\oO_X(rD) \otimes I_C$ for a
subscheme $C \subset X$ with $\dim C \le 1$
and $\hH^1(E) \in \Coh_{\le 1}(X)$. 
By the $Z_t$-stability of $E$ for $t \gg 0$, 
the sheaf $\hH^1(E)$ must be zero dimensional, 
and the following holds
\begin{align*}
\Hom(\Coh_0(X)[-1], E)=0.
\end{align*}
Therefore applying~\cite[Lemma~3.11~(iii)]{Tcurve1}, 
we see that $E$ is isomorphic to 
an object of the form
\begin{align}\label{rOF}
\oO_X(rD) \otimes (\oO_X \to F)
\end{align}
for some stable pair $(\oO_X \to F) \in P_n(X, \beta)$. 
The morphism (\ref{MP}) 
is defined by sending $E$ to $(\oO_X \to F)$. 

Conversely, we construct a morphism
\begin{align}\label{PM}
P_n(X, \beta) \to M_t(1, \alpha)
\end{align}
for $t\gg 0$. For a stable 
pair $(\oO_X \to F) \in P_n(X, \beta)$, 
the complex (\ref{rOF}) is 
an object in
$\aA_{X/Y}$ by Lemma~\ref{lem:cC}. 
We show that the object (\ref{rOF}) is $Z_t$-stable 
for $t\gg 0$. 
Let us take an exact sequence in $\aA_{X/Y}$
\begin{align*}
0 \to A \to \oO_X(rD) \otimes (\oO_X \to F)
\to B \to 0
\end{align*}
with $A, B \neq 0$. 
Then $A$ or $B$ is an object in $\bB_{\le 1}[-1]$. 
Suppose that 
$A \in \bB_{\le 1}[-1]$. We show that 
$\arg Z_t(A) <\psi$ holds for $t\gg 0$. 
If $A \notin \bB_{0}[-1]$, 
then $\arg Z_t(A)=\pi/2$, and 
the claim is obvious. 
We may assume that 
$A \in \bB_{0}[-1]$. 
If $\hH^0(A) \neq 0$, then 
it is a pure two dimensional sheaf, 
which implies $\Hom(\hH^0(A), \cC)=0$ 
by the definition of $\cC$. 
This is a contradiction as 
$\hH^0(A)$ is a subobject of (\ref{rOF}) in $\aA_{X/Y}$. 
Hence $\hH^0(A)=0$ and $A \in \tT_0[-1]$ holds. 
Then we have
\begin{align*}
\lim_{t\to \infty} \arg Z_t(A)=0<\psi. 
\end{align*} 
A similar argument shows 
that $\arg Z_t(B) >\psi$ for $t\gg 0$ if 
$B \in \bB_{\le 1}[-1]$, 
and we conclude that the object (\ref{rOF}) is 
$Z_t$-stable for $t\gg 0$. 
The morphism (\ref{rOF}) is defined
by sending $(\oO_X \to F)$ to the object
(\ref{rOF}). 
The morphisms (\ref{MP}), (\ref{PM})
are inverse each other, hence they are isomorphisms. 
\end{proof}
\begin{prop}\label{prop:PY}
For $0<t\ll 1$, we have
$M_t(1, \alpha) \cong P(\yY, -\Phi_{\ast}^{-1}(\alpha))$. 
In particular, we have 
$\DT_t(1, \alpha)=P_{-\Phi_{\ast}^{-1}(\alpha)}(\yY)$. 
\end{prop}
\begin{proof}
We first note that 
an object $E \in \aA_{X/Y}$ with 
$\rank(E)=1$ is $Z_t$-stable 
for $0<t\ll 1$ if and only if 
$E \in \Phi(\aA_{\yY})$ and it is 
$Z_0$-stable. 
This statement is an 
immediate consequence of Lemma~\ref{lem:limit}, 
together with
the fact that any rank 
one $Z_0$-semistable 
object in $\Phi(\aA_{\yY})$
is $Z_0$-stable. 
The latter fact holds since
there is no rank zero object 
$F \in \Phi(\aA_{\yY})$
with $\arg Z_0(F)=\psi$. 
Therefore 
it is enough to show that
a rank one object $E \in \Phi(\aA_{\yY})$
is $Z_0$-stable if and only if it 
is isomorphic to 
an object of the form
\begin{align}\label{PhiOF}
\Phi(\oO_{\yY} \to F)
\end{align}
for an orbifold stable pair $(\oO_{\yY} \to F)$
on $\yY$. 

Let us take a $Z_0$-stable object $E=\Phi(G)
 \in \Phi(\aA_{\yY})$
with $\rank(E)=1$. 
By the definition of $\aA_{\yY}$, there is a filtration 
in $\aA_{\yY}$
\begin{align}\label{F123}
G_1 \subset G_2 \subset G_3=G
\end{align}
satisfying the following: 
\begin{align*}
G_1 \in \Coh_{\le 1}(\yY)[-1], \
G_2/G_1=\oO_{\yY}, \ 
G_3/G_2 \in \Coh_{\le 1}(\yY)[-1].
\end{align*}
By the $Z_0$-stability of $E$, we have 
$G_3/G_2 \in \Coh_{0}(\yY)[-1]$, 
hence the sequence
\begin{align*}
0 \to \oO_{\yY} \to G_3/G_1 \to G_3/G_2 \to 0
\end{align*}
splits. This implies that we can replace the filtration (\ref{F123})
so that $G_3/G_2=0$ holds. Hence 
$G$ is isomorphic to a two term complex of the form
$(\oO_{\yY} \stackrel{s}{\to} F)$ with $F \in \Coh_{\le 1}(\yY)$. 
By the $Z_0$-stability of $E$, 
we have 
\begin{align*}
\Hom(\Coh_{0}(\yY)[-1], E)=0
\end{align*} 
hence $F$ must be a pure one dimensional 
sheaf on $\yY$. 
The $Z_0$-stability of $E$ 
also implies that 
the cokernel of $s$ is zero dimensional. 
It follows that $(\oO_{\yY} \stackrel{s}{\to} F)$
is an orbifold stable pair, and $E$ is isomorphic 
to an object of the form (\ref{PhiOF}). 

Conversely, let us take an object (\ref{PhiOF}). 
We take an exact sequence in $\Phi(\aA_{\yY})$
\begin{align*}
0 \to A \to \Phi(\oO_{\yY} \to F) \to B \to 0
\end{align*}
with $A, B \neq 0$. 
Note that either $A$ or 
$B$ is an object of
$\Phi(\Coh_{\le 1}(\yY))[-1]$.  
Suppose that $A$ is an object of 
$\Phi(\Coh_{\le 1}(\yY))[-1]$. 
Since $(\oO_{\yY} \to F)$ is an orbifold 
stable pair,
there is no non-zero morphism from 
an object in $\Phi(\Coh_0(\yY))[-1]$ to 
the object (\ref{PhiOF}). 
Therefore $A \notin \Phi(\Coh_0(\yY))[-1]$, 
which implies that
\begin{align*}
\arg Z_0(A)=\frac{\pi}{2}<\psi. 
\end{align*}
Suppose that $B$
is an object of 
$\Phi(\Coh_{\le 1}(\yY))[-1]$.
Then 
we have $B \in \Phi(\Coh_0(\yY))[-1]$
as $(\oO_{\yY} \to F)$ is an orbifold stable 
pair. 
Hence $\arg Z_0(B)>\psi$, 
and the object (\ref{PhiOF})
is $Z_0$-stable. 
\end{proof}

Finally in this subsection, we investigate
the rank zero DT type invariants: 
\begin{prop}\label{prop:DTrcm}
Suppose that $\alpha \in N_0(X/Y)$
is written as 
$\alpha=-i_{\ast} \alpha_0$
for $\alpha_0 \in K(\mathbb{P}^2)$
with $\ch(\alpha_0)=(r, c, m)$. 
If $\widehat{c}=c+r/2>0$, we have the equality for 
$t \gg 0$
\begin{align}\label{id:DT}
\DT_t(0, \alpha)=\DT(r, c, m). 
\end{align}
\end{prop}
\begin{proof}
First suppose that $r>0$ or $r=0, c>0$. 
Then a well-known 
argument shows that 
an object $E \in \bB_0[-1]$ with 
numerical class $\alpha$
is $Z_{0, t}$-semistable 
for $t\gg 0$ if and only if 
$E[1]$ is a Gieseker semistable sheaf.  
For example, the proof of this fact for 
K3 surfaces in~\cite[Proposition~6.4, Lemma~6.5]{Tst3} works 
without 
any major modification.
Also similarly to~\cite[Proposition~6.4, Lemma~6.5]{Tst3}, 
this fact also holds for every $Z_{0, t}$-stable factors of $E$. 
By the definition of $\bB_{0}$, the sheaf $E[1]$ is 
supported on $D$. Since the formal neighborhood of $D \subset X$
is isomorphic to the formal neighborhood of the zero 
section of $\omega_{\mathbb{P}^2} \to \mathbb{P}^2$, we obtain the 
identity (\ref{id:DT}). 

Next suppose that $r<0$. 
Then an 
object $E \in \bB_0[-1]$ with 
numerical class $\alpha$
is $Z_{0, t}$-semistable for $t\gg 0$
if and only if 
$\mathbb{D}(E)[1]$
is a Gieseker semistable sheaf on $\mathbb{P}^2$, 
where $\mathbb{D}(-)=\dR \hH om(-, \oO_X)$
is the dualizing functor. 
This fact also follows from 
a well-known argument: 
the object $\mathbb{D}(E)$ is numerical class
$-i_{\ast}\alpha_1$
for $\ch(\alpha_1)=-e^{-3h}(r, -c, m)$, 
and the argument of~\cite[Proposition~9.5]{TodK3}
shows that $\mathbb{D}(E)$ is also 
Bridgeland semistable near the large volume 
limit in a tilted haert. 
Hence the above argument for $r>0$ case shows
that $\mathbb{D}(E)[1]$ is 
Gieseker semistable. 
It is easy to see that the
same 
 fact 
also holds for every $Z_{0, t}$-stable factors of $E$.
Therefore we have the identity
\begin{align*}
\DT_t(0, \alpha)=\DT(e^{-3h}(r, -c, m))
\end{align*}
for $t\gg 0$. 
Then the desired identity 
follows from (\ref{eDT}) and (\ref{DT:dual}). 
\end{proof}

\subsection{Combinatorial coefficients}
In this subsection, we recall Joyce's 
combinatorial coefficients which appear
in our wall-crossing formula, and 
give their explicit description. 
For $d\in \{0, 1\}$, 
we define the following 
positive cone:
\begin{align*}
N_{\le d}^{+}(X/Y) \cneq \Imm \left(
\bB_{\le d} 
\to N_{\le d}(X/Y) \right)\setminus \{0\}. 
\end{align*}
We also define
\begin{align*}
&K^{+}(\mathbb{P}^2) \cneq 
\left\{ E \in K(\mathbb{P}^2) : 
i_{\ast}E \in N_0^{+}(X/Y) \right\} \\
&\Gamma^{+} \cneq \mathbb{Z}_{\ge 0} \oplus 
\{-N_{\le 1}^{+}(X/Y) \cup \{0\} \} \setminus \{(0, 0)\}.
\end{align*}
The cone $\Gamma^{+}$ coincides with the image of 
non-zero objects in $\aA_{X/Y}$
under the map (\ref{def:cl}).
Hence for $v\in \Gamma^{+}$, the argument 
$\arg Z_{t}(v) \in (0, \pi]$ is well-defined. 
For $v, v' \in \Gamma^{+}$, we write 
\begin{align*}
Z_{\infty}(v)\succ Z_{\infty}(v') \
(\mbox{resp. } Z_{+0}(v)\succ Z_{+0}(v'))
\end{align*}
if $\arg Z_{t}(v)> \arg Z_{t}(v')$ holds 
for $t\gg 0$
(resp.~$0< t \ll 1$). 
\begin{defi}\emph{(\cite[Definition~4.2]{Joy4})}
For 
$v_1, \cdots, v_k \in \Gamma^{+}$, 
suppose that either 
(\ref{S1}) or (\ref{S2})
holds for each $i=1, \cdots, k-1$
\begin{align}
\label{S1}
&Z_{\infty}(v_i) \preceq Z_{\infty}(v_{i+1}) \mbox{ and }
Z_{+0}(v_1+ \cdots + v_i) \succ
Z_{+0}(v_{i+1}+ \cdots + v_{k}) \\
\label{S2}
&Z_{\infty}(v_i) \succ Z_{\infty}(v_{i+1}) \mbox{ and }
Z_{+0}(v_1+ \cdots + v_i) \preceq
Z_{+0}(v_{i+1}+ \cdots + v_{k}).
\end{align}
Then define 
\begin{align}\label{Sa}
S(\{v_1, \cdots, v_k \}, Z_{\infty}, Z_{+0}) \cneq (-1)^{a}
\end{align}
where $a$ is the number of 
$i=1, \cdots, k-1$ satisfying (\ref{S1}). 
If neither (\ref{S1}) nor $(\ref{S2})$ holds for some $i$, 
we define $S(\{v_1, \cdots, v_k \}, Z_{\infty}, Z_{+0})=0$.
\end{defi}
Let $\iota(x)$ be the function
on $\mathbb{R}$ defined by
\begin{align}\label{iota}
\iota(x) \cneq \left\{ \begin{array}{cl}
1 & x>0 \\
-1 & x\le 0. 
\end{array}   \right. 
\end{align}
We can also write (\ref{Sa}) 
in the following way: 
\begin{align}\notag
\frac{1}{2^{k-1}}
\lim_{\varepsilon \to +0}
\prod_{i=1}^{k-1}
&\left\{ \iota\left( \arg Z_{1/\varepsilon}(v_i) -
\arg Z_{1/\varepsilon}(v_{i+1})
  \right) \right. \\ \label{Sa2}
&\left. -\iota\left( \arg Z_{\varepsilon}(v_1+ \cdots +v_i) -
\arg Z_{\varepsilon}(v_{i+1}+\cdots +v_k)
  \right) \right\}. 
\end{align}

\begin{defi}\emph{(\cite[Definition~4.4]{Joy4})}
For $v_1, \cdots, v_k \in \Gamma^{+}$, 
we define 
\begin{align}\notag
U(\{v_1, \cdots, v_k\}, Z_{\infty}, Z_{+0})
\cneq \sum_{1\le k'' \le k' \le k}
\sum_{\begin{subarray}{c}
\psi \colon \{1, \cdots, k\} \to \{1, \cdots, k'\} \\
\label{defi:U}
\psi' \colon \{1, \cdots, k'\} \to \{1, \cdots, k''\}
\end{subarray}} \\
\prod_{a=1}^{k''}S(\{v_i^{\dag}\}_{i \in \psi^{'-1}(a)}, Z_{\infty}, Z_{+0})
\frac{(-1)^{k^{''}-1}}{k''}
\prod_{b=1}^{k'} \frac{1}{\lvert \psi^{-1}(b) \vert !}. 
\end{align}
Here $\psi$, $\psi'$, $v_i^{\dag}$ are as follows: 
\begin{itemize}
\item $\psi$ and $\psi'$ are non-decreasing surjective maps. 
\item For $1\le i, j \le k$ with $\psi(i)=\psi(j)$, 
we have 
$Z_{\infty}(v_i)=Z_{\infty}(v_j)$. 
\item For $1\le i, j \le k''$, we have 
\begin{align}\label{m''}
Z_{+0}\left( \sum_{ a\in \psi^{-1} \psi^{'-1}(i)} v_a \right)
=Z_{+0}\left( \sum_{ a\in \psi^{-1} \psi^{'-1}(j)} v_a \right). 
\end{align}
\item The elements $v_i^{\dag} \in \Gamma^{+}$ 
for $1\le i\le k'$ are defined to be
\begin{align}\label{wi}
v_i^{\dag}=\sum_{j \in \psi^{-1}(i)} v_j. 
\end{align}
\end{itemize}
\end{defi}
We also define the 
following more explicit function:
\begin{defi}\label{defi:Urlm}
Suppose that $1\le e \le k$ and 
$(r_j, c_j, m_j) \in \mathbb{Q}^3$
for $1\le j\le k$, $j\neq e$
are given. If $\widehat{c}_j=c_j+r_j/2 >0$
holds for all $j\neq e$, we define 
\begin{align}\label{defi:U2}
&U(\{(r_j, c_j, m_j)\}_{j\neq e}) 
\cneq \\
\notag
&\lim_{\varepsilon \to +0} 
\sum_{\begin{subarray}{c}
\text{\rm{non-decreasing} }
\psi \colon \{1, \cdots, k\}
\twoheadrightarrow \{1, \cdots, k'\} \\
e' := \psi(e), \psi^{-1}(e')=\{e\}, 
\psi(i)=\psi(j) \text{ \rm{implies} }
v_i \sim v_j
\end{subarray}} 
\frac{1}{2^{k'-1}}\prod_{i=1}^{k'}
\frac{1}{\lvert \psi^{-1}(i) \rvert !}
\\
\notag
&\prod_{i<e'-1}
\left\{\iota
\left( \frac{r_{i+1}^{\dag}}{c_{i+1}^{\dag}}
-\frac{r_{i}^{\dag}}{c_{i}^{\dag}}
+\varepsilon\left( \frac{m_{i}^{\dag}}{c_{i}^{\dag}}
-\frac{m_{i+1}^{\dag}}{c_{i+1}^{\dag}} \right) \right)
 -\iota\left(\sum_{j=1}^{i} m_j^{\dag}  \right) \right\} \\
\notag
&\left\{ \iota(-r_{e'-1}^{\dag})-\iota\left(\sum_{j=1}^{e'-1}m_j^{\dag}
 \right) \right\}
\cdot \left\{ -\iota(-r_{e'+1}^{\dag})+\iota\left(\sum_{j=e'+1}^{k'}m_j^{\dag}
 \right) \right\} \\ \notag
&\prod_{i>e'}
\left\{\iota
\left( \frac{r_{i+1}^{\dag}}{c_{i+1}^{\dag}}
-\frac{r_{i}^{\dag}}{c_{i}^{\dag}}
+\varepsilon\left( \frac{m_{i}^{\dag}}{c_{i}^{\dag}}
-\frac{m_{i+1}^{\dag}}{c_{i+1}^{\dag}} \right) \right)
 +\iota\left(\sum_{j=i+1}^{k'} m_j^{\dag}  \right) \right\}. 
\end{align}
Here $\iota$ is the function (\ref{iota}), 
$v_i \sim v_j$ means $v_i=a v_j$ for some $a \in \mathbb{Q}_{>0}$, 
and 
\begin{align*}
(r_i^{\dag}, c_i^{\dag}, m_i^{\dag})
\cneq  \sum_{j\in \psi^{-1}(i)}
\left(r_j, c_j+\frac{r_j}{2}, m_j+\frac{c_j}{2}+\frac{r_j}{8} \right), \ 
1\le i\le k'. 
\end{align*}
\end{defi}
In the following lemma, 
we relate (\ref{defi:U}) with 
the function (\ref{defi:U2}):
\begin{lem}\label{lem:use}
For
$v_1, \cdots, v_k \in \Gamma^{+}$, 
suppose that there is $1\le e \le k$
such that $\rank(v_e)>0$ and 
$\rank(v_j)=0$ for $j\neq e$. 
\begin{enumerate}
\item 
If (\ref{defi:U}) is non-zero, 
then $v_j$ is written as $(0, -i_{\ast}\alpha_j)$
for $\alpha_j \in K^{+}(\mathbb{P}^2)$
satisfying $\ch(\alpha_j)
=(r_j, c_j, m_j)$ with 
$\widehat{c}_j=c_j+r_j/2>0$. 
\item In the notation of (i), we have the identity
\begin{align}\label{U=U2}
U(\{v_1, \cdots, v_k\}, Z_{\infty}, Z_{+0})
=U(\{(r_j, c_j, m_j)\}_{j\neq e}). 
\end{align}
\end{enumerate}
\end{lem}
\begin{proof}
Note that $\lim_{t \to +0}\arg Z_t(v)$
is either $0$ or $\pi/2$ or $\pi$
for any $v \in \Gamma^{+}$ with $\rank(v)=0$, 
while $\arg Z_t(v)=\psi$ for any 
$t>0$ if $\rank(v)>0$. 
This implies that 
the equality
(\ref{m''}) never happens for $i\neq j$, 
hence we have $k''=1$. 
By the same reason, 
$Z_{\infty}(v_j)=Z_{\infty}(v_e)$ never
happens for $j\neq e$, 
hence 
the map $\psi$ in (\ref{defi:U})
should satisfy $\psi^{-1} \psi(e)=\{e\}$. 
Also for rank zero $v, v' \in \Gamma^{+}$, 
it is easy to see that 
$Z_{\infty}(v)=Z_{\infty}(v')$
is equivalent to 
that $v'=a v$ for some $a \in \mathbb{Q}_{>0}$. 

(i) By the above observations, 
it is enough
to show
the claim (i) assuming 
$S(\{v_i\}_{i=1}^{k}, Z_{\infty}, Z_{+0})$
 is non-zero. 
For $j\neq e$, we have
$v_j \in N_{\le 1}(X/Y)$
as $\rank(v_j)=0$. 
Suppose that $v_j \notin N_0(X/Y)$
for some $j\neq e$. 
If $j<e$, then 
for each $j\le j'<e$, 
we have 
\begin{align*}
\frac{\pi}{2}=
\arg Z_{t}(v_1+ \cdots + v_{j'}) <\arg Z_{t}(v_{j'+1}+ \cdots
+ v_k)=\psi
\end{align*}
for any $t>0$. 
Hence $Z_{\infty}(v_{j'}) \succ Z_{\infty}(v_{j'+1})$
should hold. This implies that 
\begin{align*}
\frac{\pi}{2}
= \arg Z_{t}(v_{j}) > \arg Z_{t}(v_{e})=\psi
\end{align*}
for $t\gg 0$,  
which is a contradiction.
A similar argument also leads to a contradiction 
if $j>e$.
Hence we have
$v_j \in N_0(X/Y)$ for any $j\neq e$, 
and can write $v_j=(0, -i_{\ast}\alpha_j)$
for some $\alpha_j \in K^{+}(\mathbb{P}^2)$. 
If we write $\ch(\alpha_j)=(r_j, c_j, m_j)$, 
then $\widehat{c}_j \ge 0$ by the definition of $K^{+}(\mathbb{P}^2)$. 
If $\widehat{c}_j=0$ for some $j\neq e$, then 
$\arg Z_{0, t}(v_j)=\pi$ for any $t>0$, 
and an argument similar to above leads to a 
contradiction. Therefore $\widehat{c}_j>0$ 
for any $j\neq e$.  

(ii) 
Noting the observations before the proof of (i), the 
identity (\ref{U=U2})
is obtained by a direct computation
of $S(\{v_i^{\dag}\}_{i=1}^{k'}, Z_{\infty}, Z_{+0})$
through the description (\ref{Sa2}). 
For instance if $i < e'$, one computes
\begin{align}\notag
&\lim_{\varepsilon \to +0}
\iota\left( \arg Z_{\varepsilon}(v_1^{\dag}+ \cdots +v_i^{\dag}) -
\arg Z_{\varepsilon}(v_{i+1}^{\dag}+\cdots +v_k^{\dag})
  \right) \\
\notag
&=\lim_{\varepsilon \to +0} 
\iota \left( 
\arg \left( \sum_{j=1}^{i}
-m_j^{\dag} +\varepsilon^2 \frac{r_j^{\dag}}{2} +
\varepsilon c_j^{\dag} \sqrt{-1} \right) -\psi
    \right) \\
\notag
&=\iota\left(\sum_{j=1}^{i} m_j^{\dag}\right). 
\end{align}
The other values which appear in (\ref{Sa2})
can be computed in a similar way. 
\end{proof}

We finally define the following 
function, which appears 
as 
a wall-crossing coefficient in the 
next subsection. 
\begin{defi}\label{defi:f}
In the same situation of Definition~\ref{defi:Urlm}, 
suppose that we are also 
given $r\in \mathbb{Z}$ and $\beta \in H_2(X)$. 
We define 
\begin{align*}
&f(\{(r_j, c_j, m_j)\}_{j\neq e}, r, \beta) 
\cneq \sum_{G \in G(k)}
\frac{1}{2^{k-1}}
U(\{(r_j, c_j, m_j)\}_{j\neq e}) 
\\
&
\prod_{\begin{subarray}{c}
a \to e \text{ \rm{or} }  \\
e \to a
\text{ \rm{in} } G
\end{subarray}}
\iota(e-a) 
\left(r_a + m_a + 3rc_a -r_a D \beta
+\frac{3}{2}c_a + \frac{9}{2}r r_a
+\frac{9}{2} r^2 r_a  \right) \\
&\prod_{\begin{subarray}{c}
a \to b 
\text{ \rm{in} } G \\
a, b \neq e
\end{subarray}}
3(r_a c_b-r_b c_a).  
\end{align*}
Here $G(k)$ is the set of graphs 
defined by (\ref{def:G(k)}). 
\end{defi}

\subsection{Wall-crossing formula}
Combined with
the results so far, we prove our main result. 
\begin{thm}\label{thm:main}
Suppose that Conjecture~\ref{conj:critical}
holds. 
Then for any $\gamma \in N_{\le 1}(\yY)$, we have the 
following formula
\begin{align}\notag
P_{\gamma}(\yY)= &
\sum_{(n, \beta) \in
\mathbb{Z} \oplus H_2(X, \mathbb{Z})}
\left(\sum_{\begin{subarray}{c}
k\ge 1 \\
1\le e\le k
\end{subarray}} 
\sum_{\begin{subarray}{c}
(r_j, c_j, 2m_j)\in \mathbb{Z}^3, \
1\le j\le k, \ j\neq e, \ r \in \mathbb{Z} \\
2c_j + r_j>0, \ c_j^2 \ge 2r_j m_j,  
\text{ \rm{satisfying} } (\ref{star})
\end{subarray}}
\right. \\
\notag
&\hspace{5mm}
 (-1)^{k-1+\sum_{a \neq e}r_a + m_a + rc_a -r_a 
D \beta +\frac{3}{2}c_a + \frac{r r_a}{2}
+\frac{r^2 r_a}{2}+\sum_{a<b, 
a, b \neq e}r_a c_b -r_b c_a
} \\
\label{PY=PX}
&\hspace{20mm}
\left. f(\{(r_j, c_j, m_j)\}_{j\neq e}, r, \beta)  \cdot
 \prod_{j\neq e} \DT(r_j, c_j, m_j) \right)P_{n, \beta}(X). 
\end{align}
Here $f(\{(r_j, c_j, m_j)\}_{j\neq e}, r, \beta)$
is given in Definition~\ref{defi:f}, 
and the condition $(\ref{star})$ is 
\begin{align}\notag
(1, -\ch(\Phi_{\ast}\gamma)) &= e^{rD}(1, 0, -\beta, -n) \\
\label{star}
&-\sum_{j\neq e}
\left(0, r_j [D], \left(\frac{3}{2}r_j +c_j  \right)[l], 
\frac{3}{2}r_j + \frac{3}{2}c_j + m_j \right).
 \end{align}
\end{thm}
\begin{proof}
Under Conjecture~\ref{conj:critical}, 
we can apply the same 
wall-crossing formula in~\cite[Theorem~5.18]{JS}
to the one parameter family of weak stability conditions 
$\sigma_t$ for $t\in \mathbb{R}_{>0}$. 
For any $\gamma \in N_{\le 1}(\yY)$, we 
obtain
\begin{align}\notag
\lim_{t\to +0}\DT_t(1, -\Phi_{\ast}\gamma)
=&\lim_{t\to \infty}
\sum_{\begin{subarray}{c}
k\ge 1, v_1, \cdots, v_k \in \Gamma^{+} \\
v_1+ \cdots +v_k=(1, -\Phi_{\ast}\gamma)
\end{subarray}}
\sum_{G\in G(k)} 
\frac{(-1)^{k-1+\sum_{a<b}\chi(v_a, v_b)}}{2^{k-1}} \\
\label{WCF}
&U(\{v_1, \cdots, v_k\}, Z_{\infty}, Z_{+0})
\prod_{\begin{subarray}{c}
a \to b \\
\text{ in }G
\end{subarray}}
\chi(v_a, v_b) \prod_{j=1}^{k} \DT_t(v_j). 
\end{align}
By Lemma~\ref{lem:positive}, there is unique 
$1\le e\le k$ such that 
$\rank(v_e)=1$ and $\rank(v_j)=0$ for $j\neq e$. 
Therefore by Lemma~\ref{lem:use}, 
$v_j$ is written as 
$(0, -i_{\ast}\alpha_j)$
for $\alpha_j \in K^{+}(\mathbb{P}^2)$
with $\ch(\alpha_j)=(r_j, c_j, m_j)$, 
$c_j+r_j/2>0$. 
For $a, b \neq e$, we have
\begin{align*}
\chi(v_a, v_b)=3(r_a c_b-r_b c_a)
\end{align*}
by (\ref{Euler2}) and (\ref{iflat}). 
Also if we write $\ch(v_e)$ as	                                                $e^{rD}(1, 0, -\beta, -n)$, then 
we have
\begin{align*}
\chi(v_a, v_e)=r_a + m_a + 3rc_a -r_a D \beta
+\frac{3}{2}c_a + \frac{9}{2}r r_a
+\frac{9}{2} r^2 r_a. 
\end{align*}
The equality
$v_1 + \cdots +v_k=(1, -\Phi_{\ast}\gamma)$
is equivalent to the
equality of Chern characters
\begin{align*}
\ch(v_1)+\cdots +\ch(v_k)=(1, -\ch(\Phi_{\ast}\gamma))
\end{align*}
which coincides with the condition (\ref{star})
by (\ref{iflat}). 
Applying the above computations to (\ref{WCF}),
noting Remark~\ref{rmk:Bog}
and using Propositions~\ref{prop:DTP}, \ref{prop:PY}, \ref{prop:DTrcm} 
and Lemma~\ref{lem:use}, we obtain the desired formula. 
\end{proof}
\begin{rmk}
By the finiteness of walls in 
Proposition~\ref{prop:fund}, the 
sum in (\ref{PY=PX})
must be a finite sum. 
It is also straightforward to check that, 
for a fixed $\gamma \in N_{\le 1}(\yY)$, 
there is only a finite number of 
$(n, \beta)$, $r\in \mathbb{Z}$, 
$k\ge 1$ and 
$\{(r_j, c_j, 2m_j)\}_{j\neq e} \subset \mathbb{Z}^3$
satisfying (\ref{star}) and the following:
\begin{align*}
c_j+r_j/2>0, \ 
c_j^2 \ge 2r_j m_j, \
U(\{(r_j, c_j, m_j)\}_{j\neq e}) \neq 0. 
\end{align*} 
\end{rmk}
\begin{rmk}
If $r \neq 0$, the invariant 
$\DT(r, c, m)$ is 
computed by the recursion formula (\ref{recursion}). 
Also the invariant $\DT(0, c, m)$ for $c\neq 0$ is 
described by polynomials of 
stable pair invariants on $X$ by Lemma~\ref{rankzero}. 
Hence the formula (\ref{PY=PX})
describes $\PT_{\gamma}(\yY)$ in terms of 
polynomials of stable pair invariants on 
$X$, in principle. 
\end{rmk}
If we restrict the curve class
of the form $c[l]$ for $c>0$
and a line $l\subset D$, 
we obtain the
following relationship between 
stable pair invariants and 
generalized DT invariants on local $\mathbb{P}^2$:
\begin{cor}\label{cor:recursion}
Suppose that Conjecture~\ref{conj:critical}
holds. 
Then for any positive integer 
$c$ and $n\in \mathbb{Z}$, we have the 
following formula
\begin{align} \notag
P_{n, c[l]}(X)= &
\sum_{\begin{subarray}{c}
(n', c') \in 
\mathbb{Z}^2 \\
0\le c'<c 
\end{subarray}}
\left(\sum_{\begin{subarray}{c}
k\ge 1 \\
1\le e\le k
\end{subarray}} 
\sum_{\begin{subarray}{c}
(r_j, c_j, 2m_j)\in \mathbb{Z}^3, \ 
1\le j\le k, \ j\neq e, \ r:=\sum_{j\neq e}r_j \\
2c_j + r_j>0, \
c_j^2 \ge 2r_j m_j,  
\text{ \rm{satisfying} } (\ref{star2})
\end{subarray}}
\right. \\
& \notag
(-1)^{k+n-n'+rc'+rc+\frac{3}{2}r^3 -\frac{1}{2}r
+\sum_{a<b, 
a, b \neq e}r_a c_b -r_b c_a
}  \\
\label{rec:formula}
&\left. f(\{(r_j, c_j, m_j)\}_{j\neq e}, r, c'[l])  \cdot
 \prod_{j\neq e} \DT(r_j, c_j, m_j) \right)P_{n', c'[l]}(X). 
\end{align}
Here the condition $(\ref{star2})$ is 
\begin{align}\label{star2}
&c=c'+\frac{3}{2}r^2+r
+\sum_{j\neq e}\left( \frac{1}{2}r_j +c_j \right) \\
\notag
&n=n'-\frac{3}{2}r^3
+\left(\frac{3}{2}-3c' \right) r
+\sum_{j\neq e} \left(\frac{3}{2}c_j +m_j \right). 
 \end{align}
\end{cor}
\begin{proof}
We apply Theorem~\ref{thm:main}
for $\gamma \in N_0(\yY)$ such that
$\Phi_{\ast}\gamma \in N_0(X/Y)$ 
corresponds to $(c[l], n)$ under the map (\ref{NH2}). 
Then $P_{\gamma}(\yY)=0$ by the definition of 
orbifold stable pairs. 
In the RHS of (\ref{PY=PX}), 
there is a unique term of $r=0$ and $k=1$ which gives $P_{n, c[l]}(X)$. 
The condition (\ref{star}) determines $r$ by 
$r=\sum_{j\neq e} r_j$, and 
the condition (\ref{star}) 
for $\beta=c'[l]$
is equivalent to the condition (\ref{star2}), 
which implies $c'<c$. 
\end{proof}
\begin{rmk}
The formula (\ref{rec:formula})
is a recursion formula of stable pair invariants on 
local $\mathbb{P}^2$ in terms of generalized DT invariants
on it. 
By solving the recursion, one can
in principle  
describe the stable pair invariants on local
$\mathbb{P}^2$ in terms of generalized DT invariants
on it 
with possibly non-zero rank. 
\end{rmk}

\begin{exam}\label{exam:c-1}
Let us consider the 
$c'=c-1$ term
of (\ref{rec:formula}).
Using Bogomolov inequality in Remark~\ref{rmk:Bog}, 
we see that 
there are three kinds of 
contributions:  
\begin{enumerate}
\item $k=2$ and 
$\{(r_j, c_j, m_j)\}_{j\neq e}=\{(0, 1, n-n'-2/3)\}$
with $n'<n-1$. 
\item $k=2$ and 
$\{(r_j, c_j, m_j)\}_{j\neq e}=\{(-1, 1, 1/2)\}$. 
\item $k=3$ and 
$\{(r_j, c_j, m_j)\}_{j\neq e}=\{(1, 0, 0), (-1, 1, 1/2)\}$. 
\end{enumerate}
The corresponding invariants $\DT(r_j, c_j, m_j)$ are
given by
\begin{align}\label{comp:DT}
\DT(0, 1, \ast)=3, \ \DT(1, 0, 0)=\DT(-1, 1, 1/2)=1. 
\end{align}
By computing $f(\{(r_j, c_j, m_j)\}_{j\neq e})$, we obtain
\begin{align*}
&P_{n, c[l]}(X)= \\
&\sum_{n'<n-1}(-1)^{n-n'-1}3(n-n')P_{n', (c-1)[l]}(X)
+(-1)^{c-1}3c P_{n-3c+2, (c-1)[l]}(X) \\
&+(-9c^2+6c+3)P_{n-1, (c-1)[l]}(X)
+\left(\mbox{\rm{terms of }}P_{n', c'[l]}(X) \mbox{ \rm{with} }
c'\le c-2  \right). 
\end{align*}
In particular, 
one obtains $P_{n, [l]}(X)=3(-1)^{n-1}n$
from the above formula. 
This is compatible with the fact that $P_n(X, [l])$ is the 
$\mathbb{P}^{n-1}$-bundle over $\mathbb{P}^2$. 
\end{exam}

\begin{exam}
In the
$c'=c-2$ term
of (\ref{rec:formula}), the possible 
contributions of $\{(r_j, c_j, m_j)\}_{j\neq e}$
are as follows: 
\begin{enumerate}
\item $k=2$
and 
$\{(0, 2, m), m \in \mathbb{Z}_{\ge 0}\}$, 
$\{(-1, 2, -2)\}$, $\{(-1, 2, -1)\}$. 
\item $k=3$ and 
\begin{align*}
&\{(0, 1, m-1/2), (0, 1, m'-1/2), m, m' \in \mathbb{Z}_{>0}\}, \\
&\{(0, 1, m-1/2), (-1, 1, -1/2), m\in \mathbb{Z}_{>0}\}, \\
& \{(-2, 2, 1), (1, 0, 0)\},
\{(-2, 2, 1), (2, 0, 0)\}, 
\{(-1, 1, -1/2), (1, 1, -1/2)\}, \\
&\{(-1, 2, -2), (1, 0, 0)\},
\{(-1, 2, -1), (1, 0, 0)\},
\{(1, 1, 1/2), (-1, 1, -1/2)\}.
\end{align*}
\item $k=4$ and 
\begin{align*}
&\{(0, 1, m-1/2), (-1, 1, -1/2), (1, 0, 0), m \in \mathbb{Z}_{\ge 0} \}, \\
&\{(1, 0, 0), (-1, 1, -1/2), (-1, 1, -1/2)\}.
\end{align*}
\item $k=5$ and 
$\{(1, 0, 0), (1, 0, 0), (-1, 1, -1/2), (-1, 1, -1/2)\}$. 
\end{enumerate}
Using (\ref{eDT}), 
the invariants $\DT(r_j, c_j, m_j)$ are
computed by (\ref{comp:DT}), together with
$\DT(2, 0, 0)=1/4$, 
 $\DT(1, 0, -1)=3$, $\DT(0, 2, 0)=-6$, $\DT(0, 2, 1)=-21/4$. 
The latter two invariants are computed using 
Lemma~\ref{rankzero}. 
\end{exam}
\subsection{Constraints of stable pair invariants}
In this subsection, we derive a non-trivial 
relationship among stable pair invariants 
induced by Seidel-Thomas twist
along $\oO_D$. 
We assume that there is 
$\lL \in \Pic(X)$ such that 
$i^{\ast}\lL \cong \oO_{\mathbb{P}^2}(1)$. 
Similarly to $\dD_{X/Y}$ and $\dD_{\yY}$, we define 
triangulated categories $\dD_{X/Y}^{\lL}, \dD_{\yY}^{\lL^{\dag}}$
to be 
\begin{align*}
&\dD_{X/Y}^{\lL} \cneq \langle \lL, D^b \Coh_{\le 1}(X/Y) \rangle_{\rm{tr}}
\subset D^b \Coh(X) \\
&\dD_{\yY}^{\lL^{\dag}} \cneq \langle
\lL^{\dag}, D^b \Coh_{\le 1}(\yY) \rangle_{\rm{tr}}\subset D^b \Coh(\yY). 
\end{align*}
Here $\lL^{\dag}$ is the line bundle on $\yY$
given by (\ref{Ldag}). 
Note that the equivalence $\Phi$ 
restricts to the equivalence 
between $\dD_{\yY}^{\lL^{\dag}}$ and $\dD_{X/Y}^{\lL}$. 
We also define
\begin{align*}
&\aA_{X/Y}^{\lL} \cneq \langle \lL, \bB_{\le 1}[-1] \rangle_{\rm{ex}}
\subset \dD_{X/Y}^{\lL} \\
&\aA_{\yY}^{\lL^{\dag}} \cneq \langle \lL^{\dag}, 
\Coh_{\le 1}(\yY)[-1] \rangle_{\rm{ex}} \subset \dD_{\yY}^{\lL^{\dag}}.
\end{align*}
\begin{lem}
There are bounded t-structures on $\dD_{X/Y}^{\lL}$, 
$\dD_{\yY}^{\lL^{\dag}}$, whose hearts are 
$\aA_{X/Y}^{\lL}$, $\aA_{\yY}^{\lL^{\dag}}$ respectively. 
\end{lem}
\begin{proof}
As for $\aA_{X/Y}^{\lL}$, we follow 
the same proof of Lemma~\ref{lem:t}. 
Applying Proposition~\ref{prop:t} for 
$\dD=D^b \Coh(X)$, $L=\lL$, $\dD'=D^b \Coh_{\le 1}(X/Y)$
and $\aA'=\bB_{\le 1}[-1]$, 
it is enough to check the 
vanishings
\begin{align*}
\Hom(\lL, i_{\ast}F)=0, \ \Hom(i_{\ast}T[-1], \lL)=0
\end{align*}
for $\mu$-semistable $F \in \Coh(\mathbb{P}^2)$ with $\mu(F)\le -1/2$
and $\mu$-semistable $T \in \Coh(\mathbb{P}^2)$ with $\mu(T)>-1/2$. 
Using $i^{\ast} \lL \cong \oO_{\mathbb{P}^2}(1)$, the above 
vanishings hold by the same computation of Lemma~\ref{lem:t}. As
for $\aA_{\yY}^{\lL^{\dag}}$, the statement obvious 
follows from Lemma~\ref{lem:AY}
since $\aA_{\yY}^{\lL^{\dag}}=\aA_{\yY} \otimes \lL^{\dag}$. 
\end{proof}

We define the 
group homomorphism 
\begin{align*}
\cl^{\lL} \colon K(\dD_{X/Y}^{\lL}) \to \Gamma
\end{align*}
via the isomorphism (\ref{isom:Gamma})
in the following way:  
\begin{align*}
\cl^{\lL}(F) \cneq (\rank(F), [F]-\rank(F)[\lL]). 
\end{align*}
Let $\Gamma_{\bullet}$ be the filtration defined by 
(\ref{fGamma}),
and $Z_t$ the collection of 
group homomorphisms defined by (\ref{Z0t}). 
Similarly to Lemma~\ref{lem:sigmat}
and Lemma~\ref{lem:limit} 
one can show that
\begin{align*}
\sigma_t^{\lL} \cneq (Z_t, \aA_{X/Y}^{\lL}), \ t>0
\end{align*}
determine a one parameter family of 
weak stability conditions on $\dD_{X/Y}^{\lL}$
with respect to $(\Gamma_{\bullet}, \cl^{\lL})$
such that
\begin{align*}
\lim_{t\to +0} \sigma_t^{\lL}=
(Z_0, \Phi(\aA_{\yY}^{\lL^{\dag}})). 
\end{align*} 

\begin{prop}\label{prop:L}
(i) An object $F\in \aA_{X/Y}^{\lL}$ with $\rank(F)=0$
is $Z_t$-semistable if and only if
$F \in \aA_{X/Y}$ and it is $Z_t$-semistable. 

(ii) For $t\gg 1$, an object
$E \in \aA_{X/Y}^{\lL}$ with $\rank(E)=1$
is $Z_t$-stable for $t\gg 0$ if and only if 
$E$ is isomorphic to an object
\begin{align}\label{OrLF}
\oO_X(rD) \otimes \lL \otimes (\oO_X \to F)
\end{align}
for $r\in \mathbb{Z}$ and a stable pair $(\oO_X \to F)$ on $X$. 

(iii) For $0< t\ll 1$, 
an object $E \in \aA_{X/Y}^{\lL}$ with $\rank(E)=1$
is 
$Z_t$-stable for $0<t \ll 1$ 
if and only if $E$ is isomorphic to an object
\begin{align}\label{LdagOF}
\Phi(\lL^{\dag} \otimes (\oO_{\yY} \to F))
\end{align}
for an orbifold stable pair $(\oO_{\yY} \to F)$. 
\end{prop}
\begin{proof}
(i) is obvious. 
As for (ii), 
let $\cC^{\lL}$ be the subcategory of 
$D^b \Coh(X)$ defined by
\begin{align*}
\cC^{\lL}
\cneq \langle \lL(rD), \Coh_{\le 1}(X)[-1] : r \in \mathbb{Z}
\rangle_{\rm{ex}}. 
\end{align*}
Then we have the same statements of 
Lemma~\ref{lem:cC} and Lemma~\ref{lem:filt}
after replacing $\aA_{X/Y}$, $\cC$ 
by $\aA_{X/Y}^{\lL}$, $\cC^{\lL}$ respectively. 
Also objects of the form (\ref{OrLF}) are objects in $\cC^{\lL}$, 
and the identical argument of Proposition~\ref{prop:DTP}
shows that the objects (\ref{OrLF}) coincide with the 
set of rank one $Z_t$-stable objects for $t\gg 0$
up to isomorphisms. 
The proof of (iii)
is also identical to Proposition~\ref{prop:PY}, 
after replacing $\oO_{\yY}$ by $\lL^{\dag}$. 
\end{proof}

Similarly to $f(\{(r_j, c_j, m_j)\}_{j\neq e}, r, \beta)$, we 
set
\begin{align*}
&g(\{(r_j, c_j, m_j)\}_{j\neq e}, r, \beta) 
\cneq \sum_{G \in G(k)}
\frac{1}{2^{k-1}}
U(\{(r_j, c_j, m_j)\}_{j\neq e}) 
\\
&
\prod_{\begin{subarray}{c}
a \to e \text{ \rm{or} }  \\
e \to a
\text{ \rm{in} } G
\end{subarray}}
\iota(e-a) 
\left(m_a+\frac{1}{2}c_a +3rc_a-r_a D\beta +\frac{9}{2}r r_a +\frac{9}{2}
r^2 r_a  \right) \\
&\prod_{\begin{subarray}{c}
a \to b 
\text{ \rm{in} } G \\
a, b \neq e
\end{subarray}}
3(r_a c_b -r_b c_a).  
\end{align*}

\begin{thm}\label{thm:const}
Suppose that Conjecture~\ref{conj:critical}
holds. 
Then for any $u\in \mathbb{Q}[D] \oplus 
H^{\ge 4}(X, \mathbb{Q})$, 
we have the following formula: 
\begin{align}\label{const}
&\sum_{\begin{subarray}{c}
(n, \beta) \in \\
\mathbb{Z} \oplus H_2(X, \mathbb{Z})
\end{subarray}}
\left(\sum_{\begin{subarray}{c}
k\ge 1 \\
1\le e\le k
\end{subarray}} 
\sum_{\begin{subarray}{c}
(r_j, c_j, 2m_j)\in \mathbb{Z}^3, \
1\le j\le k, \ j\neq e, \ r \in \mathbb{Z} \\
2c_j + r_j>0, \ c_j^2 \ge 2r_j m_j, 
\text{ \rm{satisfying} } (\ref{star3})
\end{subarray}}
\right. \\
\notag&
\hspace{10mm} (-1)^{k+\sum_{j \neq e}r_j + m_j + rc_j -r_j 
D \beta +\frac{3}{2}c_j + \frac{r r_j}{2}
+\frac{r^2 r_a}{2}+\sum_{a<b, 
a, b \neq e}r_a c_b -r_b c_a
} \\
\notag&\left. \hspace{20mm} f(\{(r_j, c_j, m_j)\}_{j\neq e}, r, \beta)  \cdot
 \prod_{j\neq e} \DT(r_j, c_j, m_j) \right)P_{n, \beta}(X) \\
=\notag&
\sum_{\begin{subarray}{c}
(n, \beta) \in \\
\mathbb{Z} \oplus H_2(X, \mathbb{Z})
\end{subarray}}
\left(\sum_{\begin{subarray}{c}
k\ge 1 \\
1\le e\le k
\end{subarray}} 
\sum_{\begin{subarray}{c}
(r_j, c_j, 2m_j)\in \mathbb{Z}^3, \
1\le j\le k, \ j\neq e, \ r \in \mathbb{Z} \\
2c_j + r_j>0, \ c_j^2 \ge 2r_j m_j, 
\text{ \rm{satisfying} } (\ref{star4})
\end{subarray}}
\right. \\
\notag& \hspace{10mm} 
(-1)^{k+\sum_{j \neq e}
m_j+\frac{1}{2}c_j +rc_j+r_j D\beta +\frac{r r_j}{2}+\frac{r^2 r_j}{2}
+\sum_{a<b, 
a, b \neq e}r_a c_b -r_b c_a}
\\
\notag&\left. \hspace{20mm} g(\{(r_j, c_j, m_j)\}_{j\neq e}, r, \beta)  \cdot
 \prod_{j\neq e} \DT(r_j, c_j, m_j) \right)P_{n, \beta}(X). 
\end{align}
Here the conditions (\ref{star3}), (\ref{star4})
are
\begin{align}\label{star3}
(1, -u)= &e^{rD}(1, 0, -\beta, -n) \\
\notag
&-\sum_{j\neq e}
\left(0, r_j [D], \left(\frac{3}{2}r_j +c_j  \right)[l], 
\frac{3}{2}r_j + \frac{3}{2}c_j + m_j \right) \\
\label{star4}
e^{c_1(\lL)}-(0, \Theta_{\sharp}u)= &
e^{rD+c_1(\lL)}(1, 0, -\beta, -n) \\
\notag
&-\sum_{j\neq e}
\left(0, r_j [D], \left(\frac{3}{2}r_j +c_j  \right)[l], 
\frac{3}{2}r_j + \frac{3}{2}c_j + m_j \right) 
 \end{align}
and $\Theta_{\sharp}$ is given by (\ref{dia:Theta}). 
\end{thm}
\begin{proof}
Similarly to the proof of Theorem~\ref{thm:main}, 
we apply Joyce-Song wall-crossing formula 
to the one parameter family of weak stability 
conditions $\sigma_t^{\lL}$ for $t \in \mathbb{R}_{>0}$. 
For $t \gg 0$, the rank one invariants 
count objects of the form (\ref{OrLF}), whose 
Chern characters are given by
\begin{align*}
\ch \left(\oO_X(rD) \otimes \lL \otimes (\oO_X \to F) \right)
=e^{rD +c_1(\lL)}(1, 0, -\beta, -n)
\end{align*}
where $(\oO_X \to F) \in P_n(X, \beta)$. 
For $0<t \ll 1$, the rank 
one invariants count objects of the form (\ref{LdagOF}), 
whose Chern characters are given by 
\begin{align*}
\ch \left( \Phi(\lL^{\dag} \otimes (\oO_{\yY} \to F)) \right)
&=\ch \left( \lL \to \Phi(\lL^{\dag}\otimes F) \right) \\
&=e^{c_1(\lL)}-\ch \Phi(\lL^{\dag} \otimes F) \\
&=e^{c_1(\lL)}-(0, \Theta_{\sharp} \ch(\Phi_{\ast}\gamma))
\end{align*}
where $(\oO_{\yY} \to F) \in P(\yY, \gamma)$. 
Here we have used the diagrams (\ref{cohd2}) and (\ref{cohd3}). 

Let us take $v_a=-i_{\ast}\alpha_a \in N_0(X/Y)$
for $\alpha_a \in K^{+}(\mathbb{P}^2)$ with 
$\ch(\alpha_a)=(r_a, c_a, m_a)$ and 
$v_e \in N(X)$ with 
$\ch(v_e)=e^{rD+c_1(\lL)}(1, 0, -\beta, -n)$. 
Using (\ref{Euler2}),
we have 
\begin{align*}
\chi(v_a, v_e)=m_a+\frac{1}{2}c_a +3rc_a-r_a D\beta +\frac{9}{2}r r_a +\frac{9}{2}
r^2 r_a. 
\end{align*} 
Therefore 
for any $\gamma \in N_{\le 1}(\yY)$, 
the wall-crossing formula for $\sigma_t^{\lL}, t>0$
is described as 
\begin{align*}
P_{\gamma}(\yY)=&
\sum_{(n, \beta) \in 
\mathbb{Z} \oplus H_2(X, \mathbb{Z}) }
\left(\sum_{\begin{subarray}{c}
k\ge 1 \\
1\le e\le k
\end{subarray}} 
\sum_{\begin{subarray}{c}
(r_j, c_j, 2m_j)\in \mathbb{Z}^3, \ 
1\le j\le k, \ j\neq e, \ r\in \mathbb{Z} \\
2c_j + r_j>0, \ c_j^2 \ge 2r_j m_j, 
\text{ \rm{satisfying} } (\ref{star4})
\text{ \rm{for} } u=\ch(\Phi_{\ast}\gamma)
\end{subarray}}
\right. \\
&\hspace{10mm} (-1)^{k-1+\sum_{a \neq e}
m_a+\frac{1}{2}c_a +rc_a+r_a D\beta +\frac{r r_a}{2}+\frac{r^2 r_a}{2}
+\sum_{a<b, 
a, b \neq e}r_a c_b -r_b c_a}
\\
&\hspace{20mm} \left. g(\{(r_j, c_j, m_j)\}_{j\neq e}, r, \beta)  \cdot
 \prod_{j\neq e} \DT(r_j, c_j, m_j) \right)P_{n, \beta}(X). 
\end{align*}
Setting $u=\ch(\Phi_{\ast}\gamma)$ and comparing with 
Theorem~\ref{thm:main}, we obtain the desired identity. 
\end{proof}

\begin{rmk}\label{rmk:form}
The relation (\ref{const})
in Theorem~\ref{thm:const} is not a 
tautological relation. 
Indeed let us take $u=(0, 0, \beta_0, n_0)$. 
Then the 
conditions (\ref{star3}), (\ref{star4})
imply
\begin{align}\label{star5}
\beta&=\beta_0-\frac{3}{2}r^2[l] -r[l] -\sum_{j\neq e}\left(\frac{r_j}{2}
+c_j \right)[l] \\
\label{star6}
\beta&=\beta_0 +\frac{1}{2}\beta_0 D[l]-\frac{3}{2}r^2 [l] 
-\sum_{j\neq e} \left( \frac{r_j}{2}+c_j \right)[l]
\end{align}
respectively. 
Therefore if $\beta_0 D <0$, 
the relation (\ref{const}) is of the 
following form
\begin{align}\label{tautological}
P_{n_0, \beta_0}(X) + \sum_{n', 0<c'}(\cdots)P_{n', \beta_0-c'[l]}(X)
=\sum_{n'', 0<c''}(\cdots) P_{n'', \beta_0-c''[l]}(X) 
\end{align}
which is not tautological. 
\end{rmk}
\begin{exam}
Let $\overline{\beta}$ be a curve class on $X$
such that $\overline{\beta} \cdot D=1$ and 
$\overline{\beta}-[l]$ is not an effective 
curve class\footnote{For instance in Example~\ref{exam1}, 
one can take $D=(y_1=x_1=0)$
and $\overline{\beta}=[l']$ for a line 
$l'$ in $(y_1=x_2=0)$.}. 
 If we set $\beta_0=\overline{\beta}+[l]$, then 
the equation (\ref{star6})
does not have a solution, and the only 
solution of (\ref{star5}) is 
$\beta=\overline{\beta}$
with $\{(r_j, c_j, m_j)\}_{j\neq e}$ 
given in Example~\ref{exam:c-1}. 
Then the relation (\ref{tautological}) is computed as 
\begin{align*}
&P_{n, \overline{\beta}+[l]}(X)
= \\
&\sum_{n'<n-1}(-1)^{n-n'-1}3(n-n')P_{n', \overline{\beta}}(X)
+3 P_{n-1, \overline{\beta}}(X)
-2 P_{n, \overline{\beta}}(X). 
\end{align*}
\end{exam}
\subsection{Euler characteristic version}\label{subsec:Ever}
The 
results of Theorem~\ref{thm:main}, Corollary~\ref{cor:recursion} and 
Theorem~\ref{thm:const} are still 
conditional to Conjecture~\ref{conj:critical}. 
If we use~\cite[Theorem~6.28]{Joy4} instead of~\cite[Theorem~5.18]{JS}, 
we obtain the Euler characteristic version 
 of the above results, 
i.e. similar results after replacing
 $P_{n, \beta}(X), P_{\gamma}(\yY)$ 
by the naive Euler characteristics $\chi(P_n(X, \beta)), 
\chi(P(\yY, \gamma))$, 
without relying any conjecture. 
The proofs are the same, and
 we only give their statements. 

The following is an analogue of Theorem~\ref{thm:main}:
\begin{thm}\label{thm:main2}
For any $\gamma \in N_{\le 1}(\yY)$, we have the 
following formula
\begin{align*}
&\chi(P(\yY, \gamma))=\\
&\sum_{\begin{subarray}{c}
(n, \beta) \in \\
\mathbb{Z} \oplus H_2(X, \mathbb{Z}) 
\end{subarray}}
\left(\sum_{\begin{subarray}{c}
k\ge 1 \\
1\le e\le k
\end{subarray}} 
\sum_{\begin{subarray}{c}
(r_j, c_j, 2m_j)\in \mathbb{Z}^3, \
1\le j\le k, \ j\neq e, \ r\in \mathbb{Z} \\
2c_j + r_j>0, \ c_j^2 \ge 2r_j m_j, 
\text{ \rm{satisfying} } (\ref{star})
\end{subarray}}
(-1)^{k-1+\sum_{j \neq e}r_j + c_j + 2r_j m_j
}
\right. \\
&\left. \hspace{20mm} f(\{(r_j, c_j, m_j)\}_{j\neq e}, r, \beta)  \cdot
 \prod_{j\neq e} \DT(r_j, c_j, m_j) \right) \chi(P_n(X, \beta)). 
\end{align*}
\end{thm}
\begin{rmk}
In Theorem~\ref{thm:main2}, 
we have used the fact the Euler characteristic 
version of $\DT(r, c, m)$ differs from it 
only by a multiplication of $(-1)^{r+c+2rm+1}$
\emph{(cf.~\cite[Lemma~2.8]{TodS3})}. 
This fact will be also used in Corollary~\ref{cor:recursion2}
and Theorem~\ref{thm:const2} below. 
\end{rmk}

The following is a corollary of Theorem~\ref{thm:main2}, 
which is an analogue of Corollary~\ref{cor:recursion}: 
\begin{cor}\label{cor:recursion2}
For any positive integer 
$c$ and $n\in \mathbb{Z}$, we have the 
following formula
\begin{align*}
&\chi(P_n(X, c[l]))=\\
&\sum_{\begin{subarray}{c}
(n', c') \in 
\mathbb{Z}^2 \\ 
0\le c'<c 
\end{subarray}}
\left(\sum_{\begin{subarray}{c}
k\ge 1 \\
1\le e\le k
\end{subarray}} 
\sum_{\begin{subarray}{c}
(r_j, c_j, 2m_j)\in \mathbb{Z}^3, \
1\le j\le k, \ j\neq e, \
r:=\sum_{j\neq e}r_j
\\
2c_j + r_j>0, \
c_j^2 \ge 2r_j m_j, 
\text{ \rm{satisfying} } (\ref{star2})
\end{subarray}}
(-1)^{k+\sum_{j \neq e}r_j + c_j + 2r_j m_j} \right. \\
&\left. \hspace{20mm} f(\{(r_j, c_j, m_j)\}_{j\neq e}, r, c'[l])  \cdot
 \prod_{j\neq e} \DT(r_j, c_j, m_j) \right)
\chi(P_{n'}(X, c'[l])). 
\end{align*}
\end{cor}

Finally we have an analogue of Theorem~\ref{thm:const}:
\begin{thm}\label{thm:const2}
For any $u\in \mathbb{Q}[D] \oplus 
H^{\ge 4}(X, \mathbb{Q})$, 
we have the following formula: 
\begin{align*}
&\sum_{\begin{subarray}{c} (n, \beta) \in \\
\mathbb{Z} \oplus H_2(X, \mathbb{Z}) \end{subarray}}
\left(\sum_{\begin{subarray}{c}
k\ge 1 \\
1\le e\le k
\end{subarray}} 
\sum_{\begin{subarray}{c}
(r_j, c_j, 2m_j)\in \mathbb{Z}^3, \
1\le j\le k, \ j\neq e, \ r \in \mathbb{Z} \\
2c_j + r_j>0, \ c_j^2 \ge 2r_j m_j, 
\text{ \rm{satisfying} } (\ref{star3})
\end{subarray}}
(-1)^{k-1+\sum_{j \neq e}r_j + c_j + 2r_j m_j} \right. \\
&\left. \hspace{20mm}
 f(\{(r_j, c_j, m_j)\}_{j\neq e}, r, \beta)  \cdot
 \prod_{j\neq e} \DT(r_j, c_j, m_j) \right) \chi(P_n(X, \beta)) \\
=&
\sum_{\begin{subarray}{c}
(n, \beta) \in \\
\mathbb{Z} \oplus H_2(X, \mathbb{Z}) 
\end{subarray}}
\left(\sum_{\begin{subarray}{c}
k\ge 1 \\
1\le e\le k
\end{subarray}} 
\sum_{\begin{subarray}{c}
(r_j, c_j, 2m_j)\in \mathbb{Z}^3, \
1\le j\le k, \ j\neq e, \ r \in \mathbb{Z} \\
2c_j + r_j>0, \ c_j^2 \ge 2r_j m_j, 
\text{ \rm{satisfying} } (\ref{star4})
\end{subarray}}
(-1)^{k-1+\sum_{j \neq e}r_j + c_j + 2r_j m_j} \right.
\\
&\left. \hspace{20mm}
g(\{(r_j, c_j, m_j)\}_{j\neq e}, r, \beta)  \cdot
 \prod_{j\neq e} \DT(r_j, c_j, m_j) \right) \chi(P_n(X, \beta)). 
\end{align*}
\end{thm}

\providecommand{\bysame}{\leavevmode\hbox to3em{\hrulefill}\thinspace}
\providecommand{\MR}{\relax\ifhmode\unskip\space\fi MR }
\providecommand{\MRhref}[2]{%
  \href{http://www.ams.org/mathscinet-getitem?mr=#1}{#2}
}
\providecommand{\href}[2]{#2}

Kavli Institute for the Physics and 
Mathematics of the Universe, University of Tokyo,
5-1-5 Kashiwanoha, Kashiwa, 277-8583, Japan.

\textit{E-mail address}: yukinobu.toda@ipmu.jp


\begin{thebibliography}{MNOP06}

\bibitem[AB13]{AB}
D.~Arcara and A.~Bertram, \emph{Bridgeland-stable moduli spaces for {K}-trivial
  surfaces. {W}ith an appendix by {M}ax {L}ieblich}, J.~Eur.~Math.~Soc.~
  \textbf{15} (2013), 1--38.

\bibitem[Bay]{BaDTPT}
A.~Bayer, in preparation.

\bibitem[BBBBJ]{BBBJ}
O.~Ben-Bassat, C.~Brav, V.~Bussi, and D.~Joyce, \emph{A '{D}arboux {T}heorem'
  for shifted symplectic structures on derived {A}rtin stacks, with
  applications}, preprint, arXiv:1312.0090.

\bibitem[BCY12]{BCY}
J.~Bryan, C.~Cadman, and B.~Young, \emph{The orbifold topological vertex},
  Adv.~Math.~ \textbf{229} (2012), 531--595.

\bibitem[Bea83]{Bea}
A.~Beauville, \emph{Some remarks on {K}$\ddot{\rm{a}}$hler manifolds with
  $c_1=0$}, Classification of algebraic and analytic manifolds (Katata, 1982),
  Progress in Mathematics, vol.~39, Birkh\"auser Boston, 1983, pp.~1--26.

\bibitem[Beh97]{BGW}
K.~Behrend, \emph{Gromov-{W}itten invariants in algebraic geometry},
  Invent.~Math.~ \textbf{127} (1997), 601--617.

\bibitem[Beh09]{Beh}
\bysame, \emph{Donaldson-{T}homas invariants via microlocal geometry}, Ann.~of
  Math \textbf{170} (2009), 1307--1338.

\bibitem[BG09]{BrGr}
J.~Bryan and T.~Graber, \emph{The crepant resolution conjecture},
  Proc.~Sympos.~Pure Math.~ \textbf{80} (2009), 23--42, Algebraic
  Geometry-Seatle 2005.

\bibitem[BKR01]{BKR}
T.~Bridgeland, A.~King, and M.~Reid, \emph{The {M}c{K}ay correspondence as an
  equivalence of derived categories}, J.~Amer.~Math.~Soc.~ \textbf{14} (2001),
  535--554.

\bibitem[BM11]{BaMa}
A.~Bayer and E.~Macri, \emph{The space of stability conditions on the local
  projective plane}, Duke Math.~J.~ \textbf{160} (2011), 263--322.

\bibitem[BM13]{BrMan}
K.~Bringmann and J.~Manschot, \emph{From sheaves on $\mathbb{P}^2$ to a
  generalization of the {R}ademacher expansion}, Amer.~J.~Math.~ \textbf{135}
  (2013), 1039--1065.

\bibitem[BMT14]{BMT}
A.~Bayer, E.~Macri, and Y.~Toda, \emph{Bridgeland stability conditions on
  3-folds {I}: {B}ogomolov-{G}ieseker type inequalities}, J.~Algebraic Geom.~
  \textbf{23} (2014), 117--163.

\bibitem[Bri02]{Br1}
T.~Bridgeland, \emph{Flops and derived categories}, Invent. Math \textbf{147}
  (2002), 613--632.

\bibitem[Bri05]{Brs5}
\bysame, \emph{T-structures on some local {C}alabi-{Y}au varieties}, J.Algebra
  \textbf{289} (2005), 453--483.

\bibitem[Bri07]{Brs1}
\bysame, \emph{Stability conditions on triangulated categories}, Ann.~of Math
  \textbf{166} (2007), 317--345.

\bibitem[Bri11]{BrH}
\bysame, \emph{Hall algebras and curve-counting invariants},
  J.~Amer.~Math.~Soc.~ \textbf{24} (2011), 969--998.

\bibitem[BS]{BrSt}
J.~Bryan and D.~Steinberg, \emph{Curve counting invariants for crepant
  resolutions}, preprint, arXiv:1208.0844.

\bibitem[Bus]{Bussi}
V.~Bussi, \emph{Generalized {D}onaldson-{T}homas theory over fields ${K}\neq
  \mathbb{C}$}, preprint, arXiv:1403.2403.

\bibitem[Cala]{Cala}
J.~Calabrese, \emph{Donaldson-{T}homas invariants on {F}lops}, preprint,
  arXiv:1111.1670.

\bibitem[Calb]{Cala2}
\bysame, \emph{On the {C}repant {R}esolution {C}onjecture for
  {D}onaldson-{T}homas {I}nvariants}, preprint, arXiv:1206.6524.

\bibitem[CIT09]{CIT}
T.~Coates, H.~Iritani, and H-H. Tseng, \emph{Wall-crossings in toric
  {G}romov-{W}itten theory {I}: crepant examples}, Geometry and Topology
  \textbf{13} (2009), 2675--2744.

\bibitem[G\"90]{Got}
L.~G\"ottsche, \emph{The {B}etti numbers of the {H}ilbert scheme of points on a
  smooth projective surface}, Math.~Ann.~ \textbf{286} (1990), 193--207.

\bibitem[G\"99]{GoTheta}
\bysame, \emph{Theta functions and {H}odge numbers of moduli spaces of sheaves
  on rational surfaces}, Comm.~Math.~Phys.~ \textbf{206} (1999), 105--136.

\bibitem[HRS96]{HRS}
D.~Happel, I.~Reiten, and S.~O. Smal$\o$, \emph{Tilting in abelian categories
  and quasitilted algebras}, Mem.~Amer.~Math.~Soc, vol. 120, 1996.

\bibitem[Joy07]{Joy2}
D.~Joyce, \emph{Configurations in abelian categories {I}\hspace{-.1em}{I}.
  {R}ingel-{H}all algebras}, Advances in Math \textbf{210} (2007), 635--706.

\bibitem[Joy08]{Joy4}
D.~Joyce, \emph{Configurations in abelian categories {I}\hspace{-.1em}{V}.
  {I}nvariants and changing stability conditions}, Advances in Math
  \textbf{217} (2008), 125--204.

\bibitem[JS12]{JS}
D.~Joyce and Y.~Song, \emph{A theory of generalized {D}onaldson-{T}homas
  invariants}, Mem.~Amer.~Math.~Soc.~ \textbf{217} (2012).

\bibitem[Kaw05]{Kawlog}
Y.~Kawamata, \emph{Log crepant birational maps and derived categories},
  J.~Math.~Sci.~Univ.~Tokyo \textbf{12} (2005), 1--53.

\bibitem[KK09]{Kapu}
G.~Kapustka and M.~Kapustka, \emph{Primitive contractions of {C}alabi-{Y}au
  three folds.~{I}}, Comm.~Algebra \textbf{37} (2009), 482--502.

\bibitem[Kly91]{Kly}
A.~A. Klyachko, \emph{Moduli of vector bundles and numbers of classes},
  Functional Analysis and Its Applications \textbf{25} (1991), 67--69.

\bibitem[Kon95]{Kon}
M.~Kontsevich, \emph{Homological algebra of mirror symmetry}, Proceedings of
  {ICM}, vol.~1, Birkh$\ddot{\textrm{a}}$user, Basel, 1995.

\bibitem[Koo]{Kool}
M.~Kool, \emph{Euler characteristics of moduli spaces of torsion free sheaves
  on toric surfaces}, preprint, arXiv:0906.3393.

\bibitem[KS]{K-S}
M.~Kontsevich and Y.~Soibelman, \emph{Stability structures, motivic
  {D}onaldson-{T}homas invariants and cluster transformations}, preprint,
  arXiv:0811.2435.

\bibitem[Lie06]{LIE}
M.~Lieblich, \emph{Moduli of complexes on a proper morphism}, J.~Algebraic
  Geom.~ \textbf{15} (2006), 175--206.

\bibitem[Man]{Man3}
J.~Manschot, \emph{Sheaves on $\mathbb{P}^2$ and generalized {A}ppell
  functions}, preprint, arXiv:1407.7785.

\bibitem[Man11]{Man2}
\bysame, \emph{The {B}etti numbers of the moduli space of stable sheaves of
  rank 3 on $\mathbb{P}^2$}, Lett.~Math.~Phys.~ \textbf{98} (2011), 65--78.

\bibitem[MNOP06]{MNOP}
D.~Maulik, N.~Nekrasov, A.~Okounkov, and R.~Pandharipande,
  \emph{Gromov-{W}itten theory and {D}onaldson-{T}homas theory. {I}},
  Compositio.~Math \textbf{142} (2006), 1263--1285.

\bibitem[Mor82]{Mori}
S.~Mori, \emph{Threefolds whose canonical bundles are not numerically
  effective}, Ann.~of Math.~ \textbf{116} (1982), 133--176.

\bibitem[MRS]{MRS}
T.~Milanov, Y.~Ruan, and Y.~Shen, \emph{Gromov-{W}itten theory and cycle-valued
  modular forms}, preprint, arXiv:1206.3879.

\bibitem[OP06]{OP}
A.~Okounkov and R.~Pandharipande, \emph{Virasoro constraints for target
  curves}, Invent.~Math.~ \textbf{163} (2006), 47--108.

\bibitem[PP]{PP}
R.~Pandharipande and A.~Pixton, \emph{Gromov-{W}itten/{P}airs correspondence
  for the quintic 3-fold}, preprint, arXiv:1206.5490.

\bibitem[PT09]{PT}
R.~Pandharipande and R.~P. Thomas, \emph{Curve counting via stable pairs in the
  derived category}, Invent.~Math.~ \textbf{178} (2009), 407--447.

\bibitem[PTVV13]{PTVV}
T.~Pantev, B.~To$\ddot{\textrm{e}}$n, M.~Vaquie, and G.~Vezzosi, \emph{Shifted
  symplectic structures}, Publ.~Math.~IHES \textbf{117} (2013), 271--328.

\bibitem[Ros]{DR}
D.~Ross, \emph{Donaldson-{T}homas {T}heory and {R}esolutions of {T}oric
  {T}ransverse {A}-singularities}, preprint, arXiv:1409.7011.

\bibitem[Rua83]{Yong}
Y.~Ruan, \emph{The cohomology ring of crepant resolutions of orbifolds},
  Gromov-Witten theory of spin curves and orbifolds, Contem.~Math.~, vol. 403,
  American Mathematical Society, 1983, pp.~117--126.

\bibitem[ST01]{ST}
P.~Seidel and R.~P. Thomas, \emph{Braid group actions on derived categories of
  coherent sheaves}, Duke Math.~J.~ \textbf{108} (2001), 37--107.

\bibitem[ST11]{StTh}
J.~Stoppa and R.~P. Thomas, \emph{Hilbert schemes and stable pairs: {GIT} and
  derived category wall crossings}, Bull.~Soc.~Math.~France \textbf{139}
  (2011), 297--339.

\bibitem[Tho00]{Thom}
R.~P. Thomas, \emph{A holomorphic {C}asson invariant for {C}alabi-{Y}au 3-folds
  and bundles on ${K3}$-fibrations}, J.~Differential.~Geom \textbf{54} (2000),
  367--438.

\bibitem[Toda]{TodS3}
Y.~Toda, \emph{Generalized {D}onaldson-{T}homas invariants on the local
  projective plane}, preprint, arXiv:1405.3366.

\bibitem[Todb]{TGep}
\bysame, \emph{Gepner type stability conditions on graded matrix
  factorizations}, Algebraic Geometry (to appear), arXiv:1302.6293.

\bibitem[Tod08]{Tst3}
\bysame, \emph{Moduli stacks and invariants of semistable objects on {K}3
  surfaces}, Advances in Math \textbf{217} (2008), 2736--2781.

\bibitem[Tod09]{Tolim}
\bysame, \emph{Limit stable objects on {C}alabi-{Y}au 3-folds}, Duke Math.~J.~
  \textbf{149} (2009), 157--208.

\bibitem[Tod10a]{Tcurve1}
\bysame, \emph{Curve counting theories via stable objects~{I}: {DT/PT}
  correspondence}, J.~Amer.~Math.~Soc.~ \textbf{23} (2010), 1119--1157.

\bibitem[Tod10b]{Tolim2}
\bysame, \emph{Generating functions of stable pair invariants via
  wall-crossings in derived categories}, Adv.~Stud.~Pure Math.~ \textbf{59}
  (2010), 389--434, New developments in algebraic geometry, integrable systems
  and mirror symmetry (RIMS, Kyoto, 2008).

\bibitem[Tod12]{TodK3}
\bysame, \emph{Stable pairs on local {K}3 surfaces}, J.~Differential.~Geom.~
  \textbf{92} (2012), 285--370.

\bibitem[Tod13]{Tcurve2}
\bysame, \emph{Curve counting theories via stable objects~{II}. {DT}/nc{DT}
  flop formula}, J.~Reine Angew.~Math.~ \textbf{675} (2013), 1--51.

\bibitem[VW94]{VW}
C.~Vafa and E.~Witten, \emph{A {S}trong {C}oupling {T}est of {S}-{D}uality},
  Nucl.~Phys.~B \textbf{431} (1994).

\bibitem[Yos94]{YosB}
K.~Yoshioka, \emph{The {B}etti numbers of the moduli space of stable sheaves of
  rank 2 on $\mathbb{P}^2$}, J.~Reine Angew.~Math.~ \textbf{453} (1994),
  193--220.

\bibitem[Yos96]{Yo1}
\bysame, \emph{Chamber structure of polarizations and the moduli space of
  rational elliptic surfaces}, Int.~J.~Math.~ \textbf{7} (1996), 411--431.

\end{thebibliography}
\end{document}